\documentclass[a4paper,10pt,reqno]{amsart}
\usepackage[body={13cm,21cm}]{geometry}
\usepackage[initials]{amsrefs}

\usepackage{amssymb,amscd}
\usepackage[all]{xy}
\usepackage{nicefrac}
\usepackage{units}

\usepackage[mathscr]{eucal}
\usepackage{enumerate}

\numberwithin{equation}{section}
\theoremstyle{plain}
\newtheorem{thm}[equation]{Theorem}
\newtheorem{cor}[equation]{Corollary}
\newtheorem{lem}[equation]{Lemma}
\newtheorem{prop}[equation]{Proposition}

\newtheorem{definition}[equation]{Definition}
\newtheorem{exa}[equation]{Example}

\theoremstyle{definition}

\theoremstyle{remark}

\providecommand{\C}[1]{\mathcal{#1}}

\providecommand{\R}[1]{\mathrm{#1}}

\DeclareMathOperator{\Gal}{Gal}

\DeclareMathOperator{\tr}{Tr}

\def\Q{{\mathbb Q}}
\def\Z{{\mathbb Z}}
\def\R{{\mathbb R}}
\def\C{{\mathbb C}}

\def\G{{\mathbb G}}

\def\A{{\mathbb A}}
\def\Gal{{\rm Gal}}

\DeclareFontFamily{U}{wncy}{}
\DeclareFontShape{U}{wncy}{m}{n}{%
<5>wncyr5%
<6>wncyr6%
<7>wncyr7%
<8>wncyr8%
<9>wncyr9%
<10>wncyr10%
<11>wncyr10%
<12>wncyr6%
<14>wncyr7%
<17>wncyr8%
<20>wncyr10%
<25>wncyr10}{}
\DeclareMathAlphabet{\cyr}{U}{wncy}{m}{n}
\begin{document}

\title[Counting Integral Points in Certain Homogeneous Spaces]
{Counting Integral Points in Certain Homogeneous Spaces}

\author{Dasheng Wei$^{1,2}$}
\author{Fei Xu$^3$}

\address{$^1$ Academy of Mathematics and System Science,  CAS, Beijing
100190, P.R.China}

\address{$^2$ {Mathematisches Institute, Ludwig-Maximilians-Universit\"at M\"unchen, Theresienstr.39, 80333 M\"unchen, Germany}}

\email{dshwei@amss.ac.cn}

\address{$^3$ School of Mathematical Sciences, Capital Normal University,
Beijing 100048, P.R.China}

\email{xufei@math.ac.cn}

\date{\today}

\maketitle

 \bigskip

\section*{\it Abstract}

The asymptotic formula of the number of integral points in non-compact symmetric homogeneous spaces of semi-simple simply connected algebraic groups is given by the average of the product of the number of local solutions twisted by the Brauer-Manin obstruction. The similar result is also true for homogeneous spaces of reductive groups with some restriction. As application, we will give the explicit asymptotic formulae of the number of integral points of certain norm equations and explain that the asymptotic formula of the number of integral points in Theorem 1.1 of \cite{EMS} is equal to the product of local integral solutions over all primes and answer a question raised by Borovoi related to the example 6.3 in \cite{BR95}.

\bigskip

{\it MSC classification} : 11E72, 11G35, 11R37, 14F22, 14G25, 20G30,
20G35


\bigskip

{\it Keywords} :  integral point, homogeneous space, Tamagawa measure.

\section{Introduction} \label{sec.notation}

The Hardy-Littlewood circle method is the classical method for counting integral points. Once this method can be applied, the asymptotic formula of the number of integral points will be given by the product of the number of local solutions and the local-global principle will be true (see \cite{Sch}). However, the local-global principle can not be held in general. Recently, the existence of the integral points with Brauer-Manin obstruction on various homogeneous spaces has been studied extensively in \cite{CTX}, \cite{Ha08}, \cite{Dem}, \cite{WX}, \cite{WX1} and \cite{BD}. It is natural to ask how the asymptotic formula of the number of integral solutions will be related to Brauer-Manin obstruction for such homogeneous spaces. Indeed, such problem has already been studied by Borovoi and Rudnick in \cite{BR95} for homogeneous spaces of semi-simple groups, where they measured the difference between the number of  integral solutions and that of Hardy-Littlewood expectation by introducing so called the density functions. These density functions were described by using Kottwitz invariant in \cite{BR95}. This paper can be regarded as the continuation of \cite{BR95}.  More precisely, we show that the asymptotic formula of the integral solutions will be given by the average of the product of local solutions twisted by the Brauer-Manin obstruction.

Notation and terminology are standard if not explained. Let $F$ be a
number field, $o_F$ be the ring of integers of $F$, $\Omega_F$
be the set of all primes in $F$ and $\infty_F$ be the set all
infinite primes in $F$. We use $v<\infty_F$ to denote the
finite primes of $F$. Let $F_v$ be the completion of $F$ at
$v$ and $o_{F_v}$ be the local completion of
$o_F$ at $v$ for each $v\in \Omega_F$. Write
$o_{F_v}=F_v$ for $v\in \infty_F$. For any
finite subset $S$ of $\Omega_F$ containing $\infty_F$, the
$S$-integers are defined to be elements in $F$ which are integral
outside $S$ and denoted by $o_S$. We also denote the adeles
of $F$ by $\Bbb A_F$ and ideles of $F$ by $\Bbb I_F=\Bbb G_m(\Bbb A_F)$ and $F_{\infty}=\prod_{v\in \infty_F}F_v $.

Let $G$ be a connected reductive linear algebraic group and $\bf X$
be a separated scheme of finite type over $o_F$ whose generic fiber
$$ X={\bf X} \times_{o_F} F  \cong H \backslash G $$ is a right homogeneous space of $G$, where $H$ is the stabilizer of the fixed point $P\in X(F)$. The map induced by the fixed point $P$ is denoted by
\begin{equation} \label{induces} p: \ G \longrightarrow X  \end{equation}

The obvious necessary condition
for ${\bf X}(o_F)\neq \emptyset$ is
\begin{equation}\label{loc} \prod_{v\in \Omega_F} {\bf
X}(o_{F_v})\neq \emptyset \end{equation} which is
assumed throughout this paper. By the separatedness, one can naturally regard
${\bf X}(o_{F_v})$ as an open and compact subset of $X(F_v)$ with $v$-adic
topology for $v<\infty_F$. Since we will count the integral points by using the coordinates,  we also assume that $X$ is affine and fix the closed immersion
\begin{equation}\label{coordi} X \hookrightarrow  Spec(F[x_1, \cdots, x_n]) . \end{equation}  By Matsushima's criterion (see Theorem 4.17 in \S 4 of \cite{PV}),  the assumption (\ref{coordi}) is equivalent to say that $H$ is reductive as well. For simplicity, we will also assume that $H$ is connected. Therefore any point in $X(E)$ can be regarded as a point in $E^n$ under the closed immersion (\ref{coordi}) for any $F$-algebra $E$. A point $x\in X(F)$ can be written as $x=(z_1, \cdots, z_n)\in F^n$. Then one can define
$$ |x|_{\infty_F}= \max_{v\in \infty_F} \max_{1\leq i\leq n} \{ \ |z_i|_v  \} .$$
For $T>0$, one can set
$$N({\bf X}, T)= \sharp \{x\in {\bf X}(o_F): \ |x|_{\infty_F} \leq T\} $$ and $$X(F_\infty , T)= \{ x\in X(F_\infty): \ |x|_{\infty_F} \leq T \}. $$

Since $X$
admits a unique $G$-invariant gauge form $\omega_{X}$ up to a
scalar and $\omega_{X}$ matches with the gauge forms $\omega_H$ of $H$ and $\omega_G$ of $G$ algebraically in sense of \S 2.4 (P.24) in \cite{W}, the associated Tamagawa measures on $X(\Bbb A_F)$, $H(\Bbb A_F)$ and $G(\Bbb A_F)$ (see
Chapter II in \cite{W}) are denoted by
$$d_{X}=\prod_{v\in \Omega_F} d_v, \ \ \ \nu_{H}=\prod_{v\in \Omega_F} \nu_v \ \ \ \text{and} \ \ \ \lambda_{G}=\prod_{v\in \Omega_F} \lambda_v  $$
respectively.

Let $Br(X)=H_{et}^2(X, \Bbb G_m)$. We would like to write the Brauer-Manin paring (see \cite{Sko}) in multiplicative way
$$ \aligned   X (\Bbb A_F) \times Br(X) & \longrightarrow \mu_\infty = \varinjlim_{n} \mu_n \subset \Bbb C^\times \\
((x_v)_{v\in \Omega_F}, \alpha ) & \mapsto \prod_{v\in \Omega_F} \alpha (x_v) \endaligned
$$
where $\alpha(x_v)$'s are all roots of unity and $\alpha(x_v)=1$ for almost all $v\in \Omega_F$ by fixing the isomorphism $$\Q/\Z \xrightarrow{\cong} \mu_\infty . $$ Then one can view any element in $Br(X)$ as a locally constant $\Bbb C$-value function on $X(\Bbb A_F)$. For any subset $B$ of $Br(X)$, one can define
$$ X(\Bbb A_F)^{B} = \{ (x_v)\in X(\Bbb A_F): \ \prod_{v\in \Omega_F} \alpha (x_v)=1 \ \text{for all $\alpha\in B$} \} . $$

\begin{definition} For any $\xi\in Br(X)$,  one can define
$$N_v ({\bf X} ,  \xi)= \int_{{\bf X}(o_{F_v})} \xi
d_v  $$ for any $v< \infty_F$ and
$$ N_{\infty_F}(X, T,\xi) = \int_{X(F_\infty, T)} \xi d_{\infty_F} $$ for $T>0$.
\end{definition}

The main result of this paper is the following the asymptotic formula.

\begin{thm} \label{main} If $G$ is simply connected and almost $F$-simple such that $G(F_\infty)$ is not compact and $X=H\backslash G$ such that $H$ is the group of fixed points of an involution of $G$, then
$$ N({\bf X}, T) \sim  \sum_{\xi\in (Br(X)/Br(F))} (\prod_{v<\infty_F}  N_v({\bf X}, \xi)) \cdot N_{\infty_F} (X, T, \xi)$$
as $T\rightarrow \infty$ for any separated scheme $\bf X$ of finite type over $o_F$ with $X={\bf X}\times_{o_F} F$.
\end{thm}

 For general reductive groups, one can have the similar result with certain restriction. One of the extreme situation is that $G$ is an algebraic torus and $X$ is the trivial torsor of $G$.  Consider any character
\begin{equation} \label{chara-tor} \chi: \ \ \ G(\Bbb A_F)/G(F)St({\bf X}) \longrightarrow \Bbb C^\times  \ \ \ \text{and} \ \ \  \chi=\prod_{v\in \Omega_F} \chi_v \end{equation} where
$$ St({\bf X}) = \{ g_\Bbb A\in G(\Bbb A_F): \   (\prod_{v\in \Omega_F} {\bf
X}(o_{F_v})) \cdot g_\Bbb A = \prod_{v\in \Omega_F} {\bf
X}(o_{F_v}) \} $$
and $\chi_v$ is induced by
$$ \chi_v: G(F_v) \rightarrow  G(\Bbb A_F)/G(F)St({\bf X}) \xrightarrow{\chi} \Bbb C^\times $$ for any $v\in \Omega_F$ and $\chi_v=1$ for $v\in \infty_F$. Each $\chi_v$ also induces the locally constant function $\overline{\chi}_v$ on $X(F_v)$ by setting
$$\overline{\chi}_v (x_v) = \chi_v (p^{-1}(x_v))$$
for any $x_v\in X(F_v)$, where $p$ is the map induced by (\ref{induces}). Set
$$N_v({\bf X}, \overline{\chi}_v) = \int_{{\bf X}(o_{F_v})} \overline{\chi}_v d_v $$
for $v< \infty_F$.

\begin{thm} If $G$ is an anisotropic torus and $X$ is the trivial torsor of $G$, then
$$ N({\bf X}, T) \sim  \frac{1}{\tau(G)}(\sum_{\chi} \prod_{v<\infty_F}  N_v({\bf X}, \overline{\chi}_v) )\cdot d_{\infty_F} ( X(F_{\infty}, T))$$
as $T\rightarrow \infty$ for any separated scheme $\bf X$ of finite type over $o_F$ with $X={\bf X}\times_{o_F} F$, where $\tau(G)$ is the Tamagawa number of $G$ and $\chi$ runs over all characters in (\ref{chara-tor}).
\end{thm}

The basic idea for proving Theorem \ref{main} is to partition the orbits of the arithmetic group action in ${\bf X}(o_F)$ in proper way and apply the equi-distribution property developed in \cite{DRS}, \cite{EM} and \cite{BO}. The paper is organized as follows. We first study various orbits of ${\bf X}(o_F)$ in Section \ref{sec.orbit}. In order to obtain the similar result for general reductive groups under the assumption that the map on the $F-$points induced by (\ref{induces}) is surjective, we establish the mass formulae associated the Brauer-Manin obstruction following from \cite{X} in Section \ref{sec.mass2}. Such mass formulae with characters are initial from \cite{Kn}, \cite{We62} and \cite{SP}. In Section \ref{sec.main}, we prove our main theorems based on the results on previous sections and the equi-distribution property. As application, we first study the asymptotic formulae of the integral solutions for norm equations in Section \ref{sec.norm}. In Section \ref{sec.semisimpe}, we will explain that the asymptotic formula in Theorem 1.1 of \cite{EMS} can be given by the product of all local solutions and answer a question raised by Borovoi related to the example 6.3 in \cite{BR95} .

\bigskip

\section{orbits} \label{sec.orbit}

Fix a finite subset $S$ containing $\infty_F$ satisfying that there are the group schemes $\bf G$ and $\bf H$ of finite type over $o_S$ whose generic fibers are $G$ and $H$ respectively such that
$$ {\bf X}_S= {\bf X}\times_{o_F} o_S \cong {\bf H} \backslash {\bf G} $$ and $P\in {\bf X}(o_{F_v})$ for all $v\not\in S$.
For each $v\in S\setminus \infty_F$, one can fix a group scheme ${\bf G}_v$ of finite type over $o_{F_v}$ such that ${\bf G}_v \times_{o_{F_v}} F_v=G\times_F F_v$.

\begin{definition}\label{stlocal} For each $v\in \Omega_F$, one defines $$St({\bf X}(o_{F_v}))= \begin{cases}   {\bf G}(o_{F_v}) \ \ \ & v\not\in S \\
 \{ g\in {\bf G}_v(o_{F_v}):  \ \  {\bf X}(o_{F_v})={\bf X}(o_{F_v}) \cdot g \}  \ \ \ & v\in S\setminus \infty_F \\
 G(F_v) \ \ \ &  v\in \infty_F \end{cases} $$
\end{definition}

Then $St({\bf X}(o_{F_v}))$ is an open subgroup of $G(F_v)$
for each $v\in \Omega_F$ and it is compact for $v<\infty_F$.

\begin{lem} \label{localorb}The number of orbits
$$ [{\bf X}(o_{F_v}) / St({\bf X}(o_{F_v}))]$$  is finite.
\end{lem}

\begin{proof} If $v\in \infty_F$, then $$ \sharp [{\bf X}(o_{F_v}) / St({\bf X}(o_{F_v}))] \leq \sharp H^1(F_v, H) $$
which is finite by Theorem 6.14 and Corollary 2 of Chapter 6 in \cite{PR94}.

For any $x\in {\bf X}(o_{F_v})$ with $v<\infty_F$, the morphism induced by the point $x$
$$ f_x: \ \ \  G\cong \{x \} \times_F G \longrightarrow  X \times_F G\xrightarrow{m} X$$ is dominant and smooth.
By Proposition 3.3 of Chapter 3 in \cite{PR94}, we have that $f_x$ is an open map over $F_v$ points. This implies that $x St({\bf X}(o_{F_v}))$ is open in ${\bf X}(o_{F_v})$. By compactness of ${\bf X}(o_{F_v})$, one concludes $[{\bf X}(o_{F_v})/ St({\bf X}(o_{F_v}))]$ is finite.

\end{proof}

Define $$St({\bf X})= \prod_{v\in \Omega_F} St({\bf X}(o_{F_v}))$$
Then $St({\bf X})$ is an open subgroup of $G(\Bbb A_F)$. We set
$${\bf X} \cdot {\sigma_{\Bbb A}} := (\prod_{v\in \Omega_F}{\bf X}(o_{F_v}) \cdot  \sigma_{v})\subseteq
 X (\Bbb A_F)  $$ for any $\sigma_\Bbb A=(\sigma_v)_{v\in \Omega_F}\in
G(\Bbb A_F)$. In particular,
$${\bf X} \cdot {1_\Bbb A} =\prod_{v\in \Omega_F} {\bf X}(o_{F_v}). $$
It is clear that $\sigma_{\Bbb A}^{-1}St({\bf X})\sigma_{\Bbb A}$ acts on $ {\bf X} \cdot \sigma_{\Bbb A}$.

 \begin{cor} \label{adelorb} The number of orbits $${\bf X} \cdot \sigma_{\Bbb A}/\sigma_{\Bbb A}^{-1}St({\bf X}))\sigma_{\Bbb A} $$ is finite.

\end{cor}

\begin{proof} Write $\sigma_\Bbb A=(\sigma_v)_{v\in \Omega_F}$. There is a finite subset $S_1\supseteq S$ such that $\sigma_v\in {\bf G}(o_{F_v})$ for all $v\not \in S_1$. Then one has the component-wise bijection
$$ {\bf X} \cdot \sigma_{\Bbb A}/ \sigma_{\Bbb A}^{-1}  St({\bf X}) \sigma_{\Bbb A}
 \cong [\prod_{v\in S_1} {\bf X}(o_{F_v}) \cdot \sigma_v / \sigma_v^{-1} St({\bf X}(o_{F_v})) \sigma_v ] \times [\prod_{v\not\in S_1} {\bf X}(o_{F_v})/{\bf G}(o_{F_v})].$$

For each $v\not \in S_1$, one has the short exact sequence as pointed sets
$$ 1\rightarrow {\bf H}(o_{F_v}) \rightarrow {\bf G} (o_{F_v}) \rightarrow {\bf X}(o_{F_v}) \rightarrow H^1_{et}(o_{F_v}, {\bf H}). $$
Since $\bf H$ is connected, one has $H^1_{et}(o_{F_v}, {\bf H})=1$ by Hensel's Lemma and Lang's Theorem and $|{\bf X}(o_{F_v})/ {\bf G}(o_{F_v})|=1$ for $v\not \in S_1$.

For each $v\in S_1$, the map
$${\bf X}(o_{F_v}) \cdot \sigma_v / \sigma_v^{-1} St({\bf X}(o_{F_v})) \sigma_v \rightarrow {\bf X}(o_{F_v})/ St({\bf X}(o_{F_v})) ; \ \ \ \ \ \bar{x} \mapsto \overline{ x \cdot \sigma_v^{-1}} $$
is also bijective.  The result follows from Lemma \ref{localorb}.
\end{proof}

\begin{definition} Let $\Gamma = G(F) \cap St({\bf X})$. \end{definition}

 It is clear that $\Gamma$ acts on ${\bf X}(o_F)$.

\begin{cor} \label{globorb} The number of orbits ${\bf X}(o_F)/ \Gamma $ is finite.

\end{cor}

\begin{proof} Since the natural map
$$ {\bf X}(o_F)/  \Gamma  \longrightarrow (\prod_{v\in S} {\bf X}(o_{F_v})/ St({\bf X}(o_{F_v})) )
\times (\prod_{v\not\in S}{\bf X}(o_{F_v})/  {\bf G}(o_{F_v})) $$ is injective, the result follows from Corollary \ref{adelorb}.

\end{proof}

\begin{lem} \label{bij} For any $u_\Bbb A\in G(\Bbb A_F)$, the  natural map
$$ \aligned  &  H(F) \backslash H(\Bbb A_F) u_\Bbb A (\sigma_{\Bbb A}^{-1}St({\bf X})\sigma_{\Bbb A}) / (\sigma_{\Bbb A}^{-1}St({\bf X})\sigma_{\Bbb A})
\\
\xrightarrow{\cong}  & H(F) \backslash H(\Bbb A_F) / [H(\Bbb A_F) \cap ((u_\Bbb A\sigma_{\Bbb A}^{-1})St({\bf X})(u_\Bbb A\sigma_{\Bbb A}^{-1})^{-1})]
\endaligned $$ by sending $h_\Bbb A u_\Bbb A$ to $h_\Bbb A$ is bijective.
\end{lem}
\begin{proof} It follows from the direct verification.

\end{proof}

\begin{prop} With the map (\ref{induces}), the double coset decomposition
$$ H(F) \backslash p^{-1}( {\bf X} \cdot \sigma_\Bbb A) /(\sigma_{\Bbb A}^{-1} St({\bf X})\sigma_{\Bbb A}) $$ is finite for any given $\sigma_\Bbb A\in G(\Bbb A_F)$.

\end{prop}

\begin{proof} Since the map $p$ induces the bijection
$$  H(\Bbb A_F) \backslash p^{-1}( {\bf X} \cdot \sigma_\Bbb A) /(\sigma_{\Bbb A}^{-1} St({\bf X})\sigma_{\Bbb A}) \cong {\bf X} \cdot \sigma_\Bbb A /(\sigma_{\Bbb A}^{-1} St({\bf X})\sigma_{\Bbb A}) , $$
one only needs to show that the further double coset decomposition for each piece
$$ H(F) \backslash H(\Bbb A_F) u_\Bbb A (\sigma_{\Bbb A}^{-1}St({\bf X})\sigma_{\Bbb A}) / (\sigma_{\Bbb A}^{-1}St({\bf X})\sigma_{\Bbb A}) $$ is finite by Corollary \ref{adelorb}.

By Lemma \ref{bij}, one only needs to show the finiteness of the following double coset decomposition
$$H(F) \backslash H(\Bbb A_F) / [H(\Bbb A_F) \cap ((u_\Bbb A\sigma_{\Bbb A}^{-1})St({\bf X})(u_\Bbb A\sigma_{\Bbb A}^{-1})^{-1})]. $$ Indeed, the $v$-component of $$H(\Bbb A_F) \cap ((u_\Bbb A\sigma_{\Bbb A}^{-1})St({\bf X})(u_\Bbb A\sigma_{\Bbb A}^{-1})^{-1})$$ is an open subgroup of $H(F_v)$. Moreover, it is equal to $H(F_v)$ for $v\in \infty_F$ and is equal to ${\bf H}(\frak o_{F_v})$ for almost all $v\in \Omega_F$. The result follows from Theorem 5.1 of Chapter
5 in \cite{PR94}.
\end{proof}

\begin{definition} \label{locquiv} For any $x,y \in {\bf X}(o_F)$, we define the equivalent relation over ${\bf X}(o_F)$
$$ x\sim y \ \ \ \Leftrightarrow \ \ \ x=y \cdot s_\Bbb A $$ for some $s_\Bbb A\in St({\bf X})$. The set of the equivalent classes is denoted by ${\bf X}(o_F)/\sim $.

\end{definition}

It is clear that
$$\sharp ({\bf X}(o_F)/\sim) \leq \sharp ({\bf X}(o_F)/ \Gamma) < \infty$$
by Corollary \ref{globorb}.

\begin{prop} \label{equiv}

\item

 1).  If $G$ is semi-simple and simply connected such that $G'(F_\infty)$ is not compact for any simple factor of $G$, then the diagonal map
$$ {\bf X}(o_F)/\sim  \ \ \xrightarrow{\cong}  \ ({\bf X} \cdot 1_{\Bbb A})^{Br(X)}/St({\bf X}) $$
is bijective.
\medskip

\item
2).  If the map $G(F)\xrightarrow{p} X(F)$ in (\ref{induces}) is surjective, there is a bijection
$$  {\bf X}(o_F)/\sim  \ \xrightarrow{\cong}  \ H(\Bbb A_F) \backslash (p^{-1}({\bf X} \cdot 1_\Bbb A))  \cap  (H(\Bbb A_F) G(F)St({\bf X}))/St({\bf X}) $$

\end{prop}

\begin{proof}
\item
 1). Since $St({\bf X})$ acts on $({\bf X} \cdot 1_{\Bbb A})^{Br(X)}$ by Corollary 3.6 in \cite{BD}, it is clear that the diagonal map is injective. Since any orbit of $St({\bf X})$ is open by the proof of Lemma \ref{localorb}, there is an integral point in ${\bf X}(o_F)$ for any orbit of $St({\bf X})$ inside $({\bf X} \cdot 1_{\Bbb A})^{Br(X)}$ by the proof of Theorem 3.7 (b) in \cite{CTX}. This implies that the diagonal map is surjective.

\medskip

 \item
2).
 For any element $H(\Bbb A_F) u_\Bbb A St({\bf X})$ in the right side, there are $g\in G(F)$, $h_\Bbb A\in H(\Bbb A_F)$ and $s_\Bbb A\in St({\bf X})$ such that $u_\Bbb A=h_\Bbb A \cdot g \cdot s_\Bbb A$. Then $P\cdot g\cdot s_\Bbb A\in {\bf X} \cdot 1_\Bbb A$. Therefore $P\cdot g\in {\bf X}(o_F)$. One can define the map
$$\phi: \  H(\Bbb A_F) u_\Bbb A St({\bf X}) \mapsto [P\cdot g] $$ For the different choices of $g'\in G(F)$, $h_\Bbb A'\in H(\Bbb A_F)$ and $s_\Bbb A'\in St({\bf X})$, one has $$P\cdot g \sim P\cdot g' . $$ Therefore $\phi$ is well-defined and injective. Since the map $G(F)\xrightarrow{p} X(F)$ in (\ref{induces}) is surjective, one can verify directly that $\phi$ is bijective.
\end{proof}

Write  \begin{equation} \label{partition}{\bf X}(o_F) \  =  \ \bigcup_{i}  ({\bf X}(o_F)\cap x_i St({\bf X}))  \end{equation}  with $x_i\in {\bf X}(o_F)$.  For $$y, z \in {\bf X}(o_F)\cap x_i St({\bf X}), $$ we can define the further equivalent relation $\sim_G$ on ${\bf X}(o_F)\cap x_i St({\bf X})$ as follows
$$ y\sim_G z \ \ \Leftrightarrow \ \ y=z\cdot g $$ if there is $g\in G(F)$.

The following result is proved implicitly in the proof of Theorem 4.2 in \cite{BR95}. For convenience, we will provide the proof as well.

\begin{prop} \label{sha}
$$\sharp (({\bf X}(o_F)\cap x_i St({\bf X}))/\sim_G) \leq  \sharp (ker (H^1(F, H)\rightarrow \prod_{v\in \Omega_F} H^1(F_v, H))). $$

If $G$ is semi-simple and simply connected and $G'(F_\infty)$ is not compact for any simple factor $G'$ of $G$, then the equality holds.
\end{prop}

\begin{proof} Since the exact sequence
$$1\rightarrow H(F)\rightarrow G(F) \rightarrow X(F) \xrightarrow{\delta}  H^1(F,H) \rightarrow H^1(F,G)$$ by Galois cohomology (see Chapter I, \S 5.4 Proposition 36 in \cite{Ser}), one has that $\delta(x_i)$ is an $H$-torsor over $F$.
Since the twisting of $H$ by $\delta(x_i)$ in sense of \cite{Ser} is the stabilizer $H_{x_i}$ of $x_i$ in $G$ by the direct computation, one has the following commutative diagram
\[ \begin{CD}
H^1(F, H_{x_i})  @> {\cong}>{\circ [x_i]}>   H^1(F, H)  \\
    @VVV   @VVV  \\
\prod_{v\in \Omega_F} H^1(F_v,H_{x_i}) @>{\cong}>{\circ [x_i]}>  \prod_{v\in \Omega_F}  H^1(F_v,H)
\end{CD} \]
where the neutral element in $H^1(F, H_{x_i})$ will be sent to $\delta(x_i)$ in $H^1(F, H)$ by Chapter I, \S 5.3 Proposition 35 in \cite{Ser}. Therefore
$$\sharp (({\bf X}(o_F)\cap x_i St({\bf X}))/\sim_G) \leq  \sharp (ker (H^1(F, H_{x_i})\rightarrow \prod_{v\in \Omega_F} H^1(F_v, H_{x_i}))) $$
and
$$ \sharp (ker (H^1(F, H)\rightarrow \prod_{v\in \Omega_F} H^1(F_v, H)))=\sharp (ker (H^1(F, H_{x_i})\rightarrow \prod_{v\in \Omega_F} H^1(F_v, H_{x_i}))) $$
by the functoriality of Galois cohomology and the above diagram.

If $G$ is semi-simple and simply connected and $G'(F_\infty)$ is not compact for any simple factor $G'$ of $G$, then $G(\Bbb A_F) =St({\bf X})G(F)$ by the strong approximation for $G$ (see Theorem 7.12 of Chapter 7 in \cite{PR94}). Let $$ \xi \in ker (H^1(F, H_{x_i})\rightarrow \prod_{v\in \Omega_F} H^1(F_v, H_{x_i})).  $$ By the Hasse principle for $G$, one obtains $x\in X(F)$ such that $\delta(x)=\xi$ and $x=x_{i} \cdot g_\Bbb A$ with $g_\Bbb A\in G(\Bbb A_F)$. There are $s_\Bbb A\in St({\bf X})$ and $\sigma\in G(F)$ such that $g_\Bbb A=s_\Bbb A \sigma$. Therefore
$$ x\cdot \sigma^{-1}= x_i \cdot s_\Bbb A \in {\bf X}(o_F)\cap x_i St({\bf X})$$
and $\delta(x\cdot \sigma^{-1}) = \delta(x)= \xi $.
\end{proof}

One can further decompose $${\bf X}(o_F)\cap x_i St({\bf X})  = \bigcup_{j} (y_j^{(i)} G(F) \cap x_i St({\bf X}))= \bigcup_{j} (y_j^{(i)} G(F) \cap y_j^{(i)} St({\bf X}))$$
with $y_j^{(i)} \in {\bf X}(o_F)$. The following proposition is essentially the same as Lemma 4.4.1 (i) in \cite{BR95}, which is originally from \cite{We62}.

\begin{prop} \label{refine}  With the above notation, there is a bijection
$$ (y_j^{(i)} G(F) \cap y_j^{(i)} St({\bf X}))/\Gamma   \ \xrightarrow{\cong}  \  H_{ij}(F) \backslash (H_{ij}(\Bbb A_F)\cap G(F)St({\bf X}))/ (H_{ij}(\Bbb A_F)\cap St({\bf X}))$$ where $H_{ij}$ is the stabilizer of $y_j^{(i)}$.

In particular, if the map $G(F)\xrightarrow{p} X(F)$ in (\ref{induces}) is surjective, then there is a bijection
$$ {\bf X}(o_F)\cap x_i St({\bf X})/\Gamma \ \xrightarrow{\cong}  \  H(F) \backslash (H(\Bbb A_F)\cap G(F)St({\bf X})\sigma_i^{-1})/ (H(\Bbb A_F)\cap \sigma_i St({\bf X})\sigma_i^{-1})$$ where $\sigma_i\in G(F)$ such that $P \cdot \sigma_i=x_i$.

\end{prop}

\begin{proof} Since $$x\in y_j^{(i)} G(F) \cap y_j^{(i)} St({\bf X})  \ \ \Leftrightarrow \ \ x=y_j^{(i)}\cdot g=y_j^{(i)} \cdot s_\Bbb A $$ with $g\in G(F)$ and $s_\Bbb A\in St({\bf X})$, one has that $$ g s_\Bbb A^{-1}  \in H_{ij}(\Bbb A_F)\cap G(F)St({\bf X}) . $$ One can verify directly the map
$$ x\mapsto H_{ij}(F) g s_\Bbb A^{-1} (H_{ij}(\Bbb A_F)\cap St({\bf X}))$$ is well-defined and induces the map
$$  (y_j^{(i)} G(F) \cap y_j^{(i)} St({\bf X}))/\Gamma   \xrightarrow{\phi}  H_{ij}(F) \backslash (H_{ij}(\Bbb A_F)\cap G(F)St({\bf X}))/ (H_{ij}(\Bbb A_F)\cap St({\bf X})) $$ which is well-defined as well.

If $$H_{ij}(F) g s_\Bbb A^{-1} (H_{ij}(\Bbb A_F)\cap  St({\bf X}))=H_{ij}(F) g' {s'}_\Bbb A^{-1} (H_{ij}(\Bbb A_F)\cap St({\bf X})),$$ there is $h\in H_{ij}(F)$ and $\xi\in H_{ij}(\Bbb A_F)\cap  St({\bf X})$ such that
$g s_\Bbb A^{-1} =h  g' {s'}_\Bbb A^{-1}  \xi $. Then
$$ {g'}^{-1}h^{-1} g = {s'}_\Bbb A^{-1} \xi  s_\Bbb A \in \Gamma .$$ This implies that $\phi$ is injective. It is clear that $\phi$ is surjective.

Suppose the map $G(F)\xrightarrow{p} X(F)$ in (\ref{induces}) is surjective. Then
$$ y_j^{(i)} G(F) \cap y_j^{(i)} St({\bf X})=X(F) \cap y_j^{(i)} St({\bf X})=  X(F) \cap x_i St({\bf X}) = {\bf X}(o_F)\cap x_i St({\bf X})$$ and the result follows from applying the conjugation isomorphism given by $\sigma_i$.
\end{proof}

\bigskip

\section{mass formulae associated to elements of Brauer groups}\label{sec.mass2}

In this section, we will establish the mass formulae associated the Brauer-Manin obstruction following from \cite{X}. Such mass formulae with characters are initial from \cite{Kn}, \cite{We62} and \cite{SP}. We will keep the same notations as those in Section \ref{sec.notation} and Section \ref{sec.orbit} and  further assume that both $G$ and $H$ have no non-trivial $F$-characters. Then the Tamagawa numbers of $G$ and $H$ will be given by
$$ \tau(H)= \int_{H(F)\backslash H(\Bbb A_F)} \nu_H \ \ \ \text{and} \ \ \ \tau(G) = \int_{G(F)\backslash G(\Bbb A_F)} \lambda_G $$
respectively (see \cite{kot88}). Define $$m({\bf X}\cdot \sigma_\Bbb A ):= \nu_{\infty_F} ((H(F)\cap(\sigma_{\Bbb A}^{-1}St({\bf X})\sigma_{\Bbb A}))\backslash H(F_\infty))$$ and $$M({ \bf X}\cdot \sigma_\Bbb A ) : = \lambda_{\infty_F} ((G(F) \cap (\sigma_{\Bbb A}^{-1} St({\bf X})\sigma_{\Bbb A}))\backslash G(F_\infty)) .$$
Then both $m( {\bf X}\cdot \sigma_\Bbb A )$ and $M( { \bf X}\cdot \sigma_\Bbb A )$ are finite by Theorem 4.17 in \cite{PR94} since both $G$ and $H$ have no non-trivial $F$-characters.

Denote $$[G,{\bf X}]:= G(F)[G(\Bbb A_F), G(\Bbb A_F)]  St({\bf X}) \ \ \ \text{and} \ \ \ s({\bf X}):= [G(\Bbb A_F): [G,{\bf X}]]. $$
By Theorem 5.1 of Chapter
5 in \cite{PR94}, the index $s({\bf X})$ is finite.

\begin{definition}\label{mass}
For any given $\sigma_\Bbb A\in G(\Bbb A_F)$, we define
$$R({\bf X} \cdot \sigma_\Bbb A): = \sum_{\gamma_\Bbb A} M({\bf X}\cdot \sigma_\Bbb A \gamma_\Bbb A^{-1} )$$
where $\gamma_\Bbb A$ runs over the double cosets $G(F) \backslash [G, {\bf X}]/\sigma_\Bbb A^{-1} St({\bf X}) \sigma_\Bbb A$ and
 $$r({\bf X}\cdot \sigma_\Bbb A ): = \sum_{\gamma_\Bbb A} m({\bf X}\cdot \sigma_\Bbb A\gamma_\Bbb A^{-1})$$ where $\gamma_\Bbb A$ runs over the double coset decomposition $$ H(F) \backslash p^{-1}( {\bf X} \cdot \sigma_\Bbb A ) \cap [G, {\bf X}] /\sigma_{\Bbb A}^{-1}St({\bf X})\sigma_{\Bbb A} .  $$
\end{definition}

\begin{lem}\label{stand} For any $\sigma_\Bbb A \in G(\Bbb A_F)$, one has
$$ R({\bf X} \cdot 1_\Bbb A)=R({\bf X} \cdot \sigma_\Bbb A)= \frac{\tau(G)}{s({\bf X})} \prod_{v<\infty_F} \lambda_v (St({\bf X}(o_{F_v})) )^{-1} . $$

\end{lem}

\begin{proof}  Write the coset decomposition

$$ G(\Bbb A_F) = \bigcup_{i=1}^{s({\bf X})}  [G,{\bf X}]\tau_i .$$ Then

$$   \tau(G) = \sum_{i=1}^{s({\bf X})} \int_{G(F)\backslash [G,{\bf X}]\tau_i} \lambda_G
 =  s({\bf X}) \int_{G(F)\backslash [G,{\bf X}]} \lambda_G .  $$
Consider the double coset decomposition $$ [G,{\bf X}]=\bigcup G(F) \gamma_\Bbb A \sigma_\Bbb A^{-1} St({\bf X})\sigma_\Bbb A. $$ Then

$$\int_{G(F)\backslash [G,{\bf X}]} \lambda_G =\sum_{\gamma_\Bbb A} \int_{G(F)\backslash G(F) \gamma_\Bbb A \sigma_\Bbb A^{-1} St({\bf X}) \sigma_\Bbb A\gamma_\Bbb A^{-1}} \lambda_G  . $$ Since one has the following fundamental domain

$$ \aligned &  G(F)\backslash G(F) \gamma_\Bbb A \sigma_\Bbb A^{-1} St({\bf X}) \sigma_\Bbb A \gamma_\Bbb A^{-1}  \\  \cong & ((G(F) \cap ( \gamma_\Bbb A \sigma_{\Bbb A}^{-1} St({\bf X})\sigma_{\Bbb A}\gamma_\Bbb A^{-1} )) \backslash G(F_\infty)) \times (\prod_{v<\infty_F} \gamma_v \sigma_v^{-1}  St({\bf X}(o_{F_v})) \sigma_v \gamma_v^{-1})  \endaligned  $$ where $\sigma_\Bbb A=(\sigma_v)$ and $\gamma_\Bbb A= (\gamma_v)$,  one obtains that
$$\int_{G(F)\backslash G(F) \gamma_\Bbb A \sigma_\Bbb A St({\bf X}) \sigma_\Bbb A^{-1} \gamma_\Bbb A^{-1}} \lambda_G =M({\bf X}\cdot \sigma_\Bbb A \gamma_\Bbb A^{-1})\prod_{v< \infty_F} \lambda_v(St({\bf X}(o_{F_v})))$$ by Theorem 5.5 in Chapter 5 of \cite{PR94} and the assumption that $G$ has no non-trivial $F$-characters. Combining the above together, one completes the proof.
\end{proof}

The following proposition gives the arithmetic interpretation about $r({\bf X}\cdot 1_\Bbb A )$.

\begin{prop} \label{arith} If $[G,{\bf X}]= G(F)St({\bf X})$ and the map $G(F)\xrightarrow{p} X(F)$ in (\ref{induces}) is surjective, then  $$ r({\bf X} \cdot 1_{\Bbb A})= \sum_i  m({\bf X} \cdot \tau_i^{-1}) $$
where $$ {\bf X}(o_F) = \bigcup_{i} \  y_i\cdot \Gamma  \ \ \ \text{and} \ \ \ P \cdot \tau_i =y_i$$ for  $\tau_i\in G(F)$.
\end{prop}

\begin{proof} Since the natural map
$$ H(F) \backslash p^{-1} ({\bf X} \cdot 1_{\Bbb A}) \cap [G, {\bf X}]/ St({\bf X}) \rightarrow
H(\Bbb A_F) \backslash p^{-1} ({\bf X} \cdot 1_{\Bbb A}) \cap  (H(\Bbb A_F) [G, {\bf X}])/ St({\bf X}) $$ is surjective and the fiber of an element $$H(\Bbb A_F) \tau St({\bf X}) \ \ \text{ with $\tau\in G(F)$ and $P\cdot \tau\in {\bf X}(o_F)$} $$  is given by $$ \aligned & H(F) \backslash (H(\Bbb A_F) \tau St({\bf X})  \cap G(F) St({\bf X}) ) /St({\bf X})  \\
\cong &  H(F) \backslash (H(\Bbb A_F)\cap G(F)St({\bf X})\tau^{-1})/ (H(\Bbb A_F)\cap \tau St({\bf X})\tau^{-1}) \endaligned $$ by Lemma \ref{bij}, the result follows from Prop.\ref{equiv}, 2) and Prop.\ref{refine}.
\end{proof}

The assumption that $[G,{\bf X}]= G(F)St({\bf X})$ will be satisfied if any simple factor of semi-simple part of $G$ is not compact by Theorem 7.28 in \cite{Dem}.  Let $$h({\bf X}):=[H(\Bbb A_F): H(\Bbb A_F) \cap [G,{\bf X}]].$$

\begin{prop}\label{explicit} If $$ p^{-1}( {\bf X} \cdot  \sigma_\Bbb A ) = \bigcup_{u_\Bbb A}  H(\Bbb A_F) u_\Bbb A (\sigma_{\Bbb A}^{-1} St({\bf X})\sigma_{\Bbb A}) $$
for $\sigma_\Bbb A\in G(\Bbb A_F)$, then
$$r({\bf X} \cdot \sigma_\Bbb A )= \frac{\tau(H)}{h({\bf X})} \sum_{u_\Bbb A} (\prod_{v<\infty_F} \nu_v  (H(F_v) \cap (u_v \sigma_v^{-1}  St({\bf X}(o_{F_v})) \sigma_v u_v^{-1}))^{-1}) $$
where $u_\Bbb A=(u_v)$ runs over the above double coset decomposition satisfying $u_\Bbb A\in H(\Bbb A_F)[G,{\bf X}]$.
\end{prop}

\begin{proof} The contribution of each piece $H(\Bbb A_F) u_\Bbb A (\sigma_{\Bbb A}^{-1} St({\bf X})\sigma_{\Bbb A})$ to $r({\bf X} \cdot \sigma_\Bbb A)$ is given by
$$g:=\sum_{\gamma_\Bbb A} m({\bf X}\cdot \sigma_\Bbb A\gamma_\Bbb A^{-1})$$ where $\gamma_\Bbb A$ runs over the double coset decomposition $$ H(F) \backslash H(\Bbb A_F) u_\Bbb A (\sigma_{\Bbb A}^{-1} St({\bf X})\sigma_{\Bbb A}) /\sigma_{\Bbb A}^{-1}St({\bf X})\sigma_{\Bbb A} \ \ \ \text{and} \ \ \ \gamma_\Bbb A\in [G,{\bf X}] . $$
By Lemma \ref{bij}, one has
$$g=\sum_{\gamma_\Bbb A} m({\bf X}\cdot \sigma_\Bbb A u_\Bbb A^{-1} \gamma_\Bbb A^{-1}) $$ where $\gamma_\Bbb A$ runs over the double coset decomposition
 $$H(F) \backslash (H(\Bbb A_F)\cap u_\Bbb A^{-1}[G,{\bf X}]) / (H(\Bbb A_F) \cap ((u_\Bbb A\sigma_{\Bbb A}^{-1})St({\bf X})(u_\Bbb A\sigma_{\Bbb A}^{-1})^{-1})) $$ and
 $$ g \neq 0 \ \ \ \Leftrightarrow \ \ \ u_\Bbb A \in H(\Bbb A_F) [G, {\bf X}] . $$

For any $\gamma_\Bbb A\in H(\Bbb A_F)$, one has the following fundamental domain

$$ \aligned &  H(F)\backslash H(F) (H(\Bbb A_F)\cap  (\gamma_\Bbb A u_\Bbb A \sigma_\Bbb A^{-1} St({\bf X}) \sigma_\Bbb A u_\Bbb A^{-1} \gamma_\Bbb A^{-1}) ) \\  \cong & ((H(F) \cap ( \gamma_\Bbb A u_\Bbb A  \sigma_{\Bbb A}^{-1} St({\bf X})\sigma_{\Bbb A} u_\Bbb A^{-1} \gamma_\Bbb A^{-1} )) \backslash H(F_\infty)) \\
& \times  \prod_{v<\infty_F} (H(F_v) \cap (\gamma_v u_v \sigma_v^{-1}  St({\bf X}(o_{F_v})) \sigma_v u_v^{-1} \gamma_v^{-1}))  \endaligned  $$ where $\sigma_\Bbb A=(\sigma_v)$, $u_\Bbb A=(u_v)$ and $\gamma_\Bbb A= (\gamma_v)$. Then
$$ \aligned   m({\bf X}\cdot \sigma_\Bbb A u_\Bbb A^{-1} \gamma_\Bbb A^{-1})= & \nu_H( H(F)\backslash H(F) (H(\Bbb A_F)\cap  (\gamma_\Bbb A u_\Bbb A \sigma_\Bbb A^{-1} St({\bf X}) \sigma_\Bbb A u_\Bbb A^{-1} \gamma_\Bbb A^{-1}) ) \\
& \cdot \prod_{v<\infty_F} \nu_v  (H(F_v) \cap (\gamma_v u_v \sigma_v^{-1}  St({\bf X}(o_{F_v})) \sigma_v u_v^{-1} \gamma_v^{-1}))^{-1} \\
= & \nu_H (H(F)\backslash H(F) \gamma_\Bbb A (H(\Bbb A_F)\cap  ( u_\Bbb A \sigma_\Bbb A^{-1} St({\bf X}) \sigma_\Bbb A u_\Bbb A^{-1}) ) \\
& \cdot \prod_{v<\infty_F} \nu_v  (H(F_v) \cap (u_v \sigma_v^{-1}  St({\bf X}(o_{F_v})) \sigma_v u_v^{-1}))^{-1}.
\endaligned  $$
Therefore $$ \aligned g=& \nu_H (H(F)\backslash (H(\Bbb A_F)\cap u_\Bbb A^{-1}[G,{\bf X}])) \prod_{v<\infty_F} \nu_v  (H(F_v) \cap (u_v \sigma_v^{-1}  St({\bf X}(o_{F_v})) \sigma_v u_v^{-1}))^{-1} \\ = & \frac{\tau(H)}{h({\bf X})} \prod_{v<\infty_F} \nu_v  (H(F_v) \cap (u_v \sigma_v^{-1}  St({\bf X}(o_{F_v})) \sigma_v u_v^{-1}))^{-1}   \endaligned $$ and the proof is complete.

\end{proof}

\begin{cor} \label{weldefine} For any $\sigma_\Bbb A\in G(\Bbb A_F)$ and $\tau_\Bbb A\in H(\Bbb A_F) [G, {\bf X}]$, one has $$r({\bf X} \cdot \sigma_\Bbb A)=r({\bf X} \cdot \sigma_\Bbb A \tau_\Bbb A) . $$
\end{cor}
\begin{proof} If
 $$ p^{-1}( {\bf X} \cdot  \sigma_\Bbb A ) = \bigcup_{u_\Bbb A}  H(\Bbb A_F) u_\Bbb A (\sigma_{\Bbb A}^{-1} St({\bf X})\sigma_{\Bbb A}), $$
then
$$ p^{-1}( {\bf X} \cdot  \sigma_\Bbb A \tau_\Bbb A ) = \bigcup H(\Bbb A_F) u_\Bbb A \tau_\Bbb A ((\sigma_{\Bbb A}\tau_\Bbb A)^{-1} St({\bf X})\sigma_{\Bbb A}\tau_\Bbb A). $$ Since $$u_\Bbb A\in  H(\Bbb A_F)[G,{\bf X}] \ \ \ \Leftrightarrow \ \ \  u_\Bbb A \tau_\Bbb A \in H(\Bbb A_F)[G,{\bf X}] , $$
one obtains the result by replacing $u_\Bbb A$ with $u_\Bbb A \tau_\Bbb A$ and $\sigma_\Bbb A$ with $\sigma_\Bbb A \tau_\Bbb A$ in Prop. \ref{explicit}.
\end{proof}

Let

\begin{equation} \label{chara-brau}
\chi: \ \ \ G(\Bbb A_F)/H(\Bbb A_F)[G,{\bf X}] \longrightarrow \Bbb C^\times
\end{equation}
be a character. Then $$\chi=\prod_{v\in \Omega_F} \chi_v$$ where $\chi_v$ is induced by
$$ \chi_v: G(F_v) \rightarrow  G(\Bbb A_F)/H(\Bbb A_F)[G,{\bf X}] \xrightarrow{\chi} \Bbb C^\times $$ for any $v\in \Omega_F$ and $\chi_v=1$ for $v\in \infty_F$. Moreover, each $\chi_v$ also induces the locally constant function $\overline{\chi}_v$ on $X(F_v)$ by setting
$$\overline{\chi}_v (x_v) = \chi_v (p^{-1}(x_v))$$
for any $x_v\in X(F_v)$, where $p$ is the map induced by (\ref{induces}). This $\overline{\chi}_v$ is well-defined since $\chi_v$ is trivial over $H(F_v)$.

Define
$$N_v({\bf X}, \overline{\chi}_v) = \int_{{\bf X}(o_{F_v})} \overline{\chi}_v d_v $$
for $v< \infty_F$. One can establish the mass formula associated to $\chi$.

\begin{thm}\label{premain} For any $\chi$ in (\ref{chara-brau}), one has
$$\frac{1}{R({\bf X} \cdot 1_\Bbb A)}\sum_{\sigma_\Bbb A \in G(\Bbb A_F)/H(\Bbb A_F)[G,{\bf X}]} \chi(\sigma_\Bbb A) \cdot r({\bf X} \cdot \sigma_\Bbb A^{-1})= \frac{\tau(H) s({\bf X})}{\tau(G) h({\bf X})} \prod_{v<\infty_F} N_v ({\bf X},\overline{\chi}_v) .$$
\end{thm}

\begin{proof} By Corollary \ref{weldefine}, the left hand side is well-defined.
Let
$$ p^{-1}({\bf X} \cdot 1_\Bbb A) = \bigcup_{u_\Bbb A} H(\Bbb A_F) u_\Bbb A St({\bf X}). $$ Then
$$ p^{-1} ({\bf X} \cdot \sigma_\Bbb A^{-1}) = \bigcup_{u_\Bbb A} H(\Bbb A_F) u_\Bbb A\sigma_\Bbb A^{-1} (\sigma_\Bbb A St({\bf X}) \sigma_\Bbb A^{-1})  $$  and
$$r({\bf X} \cdot \sigma_\Bbb A^{-1} )= \frac{\tau(H)}{h({\bf X})} \sum_{u_\Bbb A} (\prod_{v<\infty_F} \nu_v  (H(F_v) \cap (u_v  St({\bf X}(o_{F_v})) u_v^{-1}))^{-1}) $$
with $u_\Bbb A \sigma_\Bbb A^{-1} \in H(\Bbb A_F)[G,{\bf X}]$ by Prop.\ref{explicit}. Since the measures $$d_{X}=\prod_{v\in \Omega_F}d_v, \ \ \ \nu_{H}=\prod_{v\in \Omega_F} \nu_v \ \ \ \text{ and } \ \ \  \lambda_{G}=\prod_{v\in \Omega_F} \lambda_v  $$ match together, one has
$$\lambda_v (u_v  St({\bf X}(o_{F_v})) u_v^{-1})= d_v (P\cdot(u_v  St({\bf X}(o_{F_v})) u_v^{-1})) \cdot  \nu_v (H(F_v) \cap (u_v  St({\bf X}(o_{F_v})) u_v^{-1}))$$ for $v<\infty_F$. Therefore
$$\nu_v (H(F_v) \cap (u_v  St({\bf X}(o_{F_v})) u_v^{-1}))^{-1}= d_v(P\cdot u_v St({\bf X}(o_{F_v}))) \cdot \lambda_v (St({\bf X}(o_{F_v})))^{-1}$$
for $v<\infty_F$. This implies that
$$ \chi (\sigma_\Bbb A) \cdot r({\bf X} \cdot \sigma_\Bbb A^{-1})= \frac{\tau(H) s({\bf X})}{\tau(G) h({\bf X})}R({\bf X} \cdot 1_\Bbb A) \sum_{u_\Bbb A} \prod_{v<\infty_F} \int_{P\cdot u_v St({\bf X}(o_{F_v}))} \overline{\chi}_v d_v  $$ by Lemma \ref{stand} with $u_\Bbb A \in H(\Bbb A_F)[G,{\bf X}]\sigma_\Bbb A$. One concludes that $$ \sum_{\sigma_\Bbb A\in G(\Bbb A_F)/H(\Bbb A_F)[G,{\bf X}]}  \chi (\sigma_\Bbb A) \cdot r({\bf X} \cdot \sigma_\Bbb A^{-1})= \frac{\tau(H) s({\bf X})}{\tau(G) h({\bf X})}R({\bf X} \cdot 1_\Bbb A) \prod_{v<\infty_F} \int_{{\bf X}(o_{F_v})} \overline{\chi}_v d_{v} $$ and the proof is complete.
\end{proof}

\begin{cor} \label{key}  $$r({\bf X} \cdot 1_\Bbb A)=\frac{\tau(H)}{\tau(G)}R({\bf X} \cdot 1_\Bbb A) \sum_{\chi} (\prod_{v<\infty_F} N_v({\bf X}, \overline{\chi}_v))$$ where $\chi$ runs over all characters in (\ref{chara-brau}).
\end{cor}
\begin{proof} It follows from
$$ [G(\Bbb A_F): H(\Bbb A_F)[G,{\bf X}]] = \frac{s({\bf X})}{h({\bf X})} $$ and
$$ \sum_{\chi}  \chi (\sigma_\Bbb A)= \begin{cases}  [G(\Bbb A_F): H(\Bbb A_F)[G,{\bf X}]] \ \ \ & \text{if $\sigma_\Bbb A \in H(\Bbb A_F)[G,{\bf X}]$}  \\
0 \ \ \ & \text{otherwise} \end{cases}
$$ where $\chi$ runs over all characters in (\ref{chara-brau}).
\end{proof}

The immediate application of the above result is to test the existence of the integral points on $\bf X$.

\begin{cor} If $[G,{\bf X}]= G(F)St({\bf X})$ and the map $G(F)\xrightarrow{p} X(F)$ in (\ref{induces}) is surjective, then
$${\bf X}(o_F) \neq \emptyset \ \  \Leftrightarrow \ \ \{ (x_v) \in \prod_{v<\infty_F} {\bf X}(o_{F_v}):  \prod_{v<\infty_F} \overline{\chi}_v(x_v)=1, \  \forall \chi \ \text{in} \ (\ref{chara-brau})\} \neq \emptyset .$$
\end{cor}
\begin{proof} By Proposition \ref{arith},  one has
$$ {\bf X}(o_F) = \emptyset \ \ \Leftrightarrow \ \ r({\bf X} \cdot 1_\Bbb A)=0 .$$ Therefore
$$ {\bf X}(o_F) = \emptyset \ \ \Leftrightarrow \ \ \sum_{\chi} \prod_{v<\infty_F}  N_v({\bf X}, \overline{\chi}_v) =0  $$
by Corollary \ref{key}. Decompose
$$ \prod_{v<\infty_F} {\bf X}(o_{F_v}) = D_0 \cup D_1 \cup \cdots \cup D_l $$ as disjoint closed subsets such that $\chi|_{D_i}$ is constant for all $\chi$ with $0\leq i\leq l$ and  $$D_0 =\{ (x_v) \in \prod_{v<\infty_F} {\bf X}(o_{F_v}):  \prod_{v<\infty_F} \overline{\chi}_v(x_v)=1, \  \forall \chi \ \text{in} \ (\ref{chara-brau})\}. $$
For any $i\neq 0$, there is $\chi$ in (\ref{chara-brau}) such that $\chi(D_i)\neq 1$. This implies that
$$\sum_{\chi}\int_{D_i}\chi d_{X} = (\sum_{\chi} \chi(D_i)) \int_{D_i} d_{X} =0.  $$ Therefore
$$\sum_{\chi} \prod_{v<\infty_F}  N_v({\bf X}, \overline{\chi}_v) = \frac{s({\bf X})}{h({\bf X})}\int_{D_0} d_{X} .$$ One concludes that
$$ {\bf X}(o_F) = \emptyset \ \ \Leftrightarrow \ \ \int_{D_0} d_{X} =0 .  $$
Since $D_0$ is an open compact subset of $\prod_{v<\infty_F} {\bf X}(o_{F_v})$, one has
$$ \int_{D_0} d_{X} =0 \ \ \Leftrightarrow \ \ D_0=\emptyset $$
and the proof is complete.
\end{proof}

To conclude this section, we will point out that all characters in (\ref{chara-brau}) can be interpreted as the elements in $Br(X)$. Indeed, one can view $\chi$ as an element $\xi$ in $Br(X)$ satisfying
$$  \xi(x_\Bbb A \cdot s_\Bbb A)= \xi (x_\Bbb A) $$
for all $x_\Bbb A\in X(\Bbb A_F)$ and $s_\Bbb A\in St({\bf X})$ by Theorem 8.2 in \cite{Dem}, Theorem 2.8 and Corollary 3.5 in \cite{BD}.

\bigskip

\section{Counting the integral points via equi-distribution}   \label{sec.main}

We keep the same notation as that in the previous sections.

\begin{lem} \label{average} Suppose $X=H\backslash G$ where both $G$ and $H$ are the connected reductive groups without non-trivial characters over $F$. For any finite subgroup $B$ of $Br(X)/Br(F)$, one has
$$ \sum_{\xi \in B} (\prod_{v<\infty_F} \int_{{\bf X}(o_{F_v})} \xi d_v \cdot \int_{X(F_\infty, T)} \xi d_{\infty_F}) =  \sharp B  \int_{(\prod_{v<\infty_F} {\bf X}(o_{F_v}) \times X(F_\infty, T))^B} d_{X} $$
where
$$(\prod_{v<\infty_F} {\bf X}(o_{F_v}) \times X(F_\infty, T))^B= (\prod_{v<\infty_F} {\bf X}(o_{F_v}) \times X(F_\infty, T))\cap ({\bf X} \cdot 1_{\Bbb A} )^B  $$ and $\bf X$ is a separated scheme of finite type over $o_F$ such that ${\bf X}\times_{o_F} F=X$. \end{lem}

\begin{proof} Since $G$ and $H$ have no non-trivial $F$-characters, one concludes that the infinite product in the above lemma is convergent. Decompose $$ \prod_{v<\infty_F} {\bf X}(o_{F_v}) \times X(F_\infty, T)= \bigcup_i X_i $$ as a finite disjoint union such that all $X_i$'s are all closed and $\xi$ takes the constant value on $X_i$ for all $\xi\in B$.

If there is $\xi\in B$ such that $\xi$ does not take the trivial value on $X_i$, then
$$\sum_{\xi \in B} \int_{X_i} \xi d_{X} = (\sum_{\xi \in B} \xi(X_i)) \int_{X_i} d_{X} = 0 $$ and the proof is complete.
\end{proof}

We further assume the following equi-distribution property

\begin{equation} \label{equi-dis}
\sharp \{y \in  x \cdot \Gamma:  |y|_{\infty_F}\leq  T  \} \sim  \frac{\nu_{x, \infty_F}( \Gamma_{H_x}\backslash H_x(F_\infty)) }{\lambda_{\infty_F} (\Gamma \backslash G(F_\infty)) } d_{\infty_F} (x \cdot G(F_{\infty}) \cap X(F_\infty,  T))
\end{equation}
as $T \rightarrow \infty$, where $H_x$ is the stabilizer of $x$ in $G$, $\Gamma_{H_x}=H_x(F) \cap \Gamma$ and $\nu_{x, \infty_F}$ is the induced Tamagawa measure over $H_x$. This assumption has been proved in the following various situations.

1).  $G$ is an anisotropic torus and $X$ is a trivial torsor of $G$. The proof is exactly similar to the proof of Dirichlet Unit Theorem, which is given in \cite{Shy} (see also Theorem 5.12 in \cite{PR94}).

2). $G$ is simply connected, almost $F$-simple, $G(F_\infty)$ is not compact and $X$ is symmetric. Such result was first proved in \cite{DRS}, soon simplified in \cite{EM} over $\Bbb Q$ and extended to general number fields in \cite{BO}.

3). $H$ is a maximal proper connected reductive group over $\Bbb Q$. Such result was proved in \cite{EMS}, which extends the result in \cite{DRS} and \cite{EM} over $\Bbb Q$.

\begin{thm}\label{semisimple} Suppose $X=H\backslash G$ where $G$ is semi-simple and simply connected and $G'(F_\infty)$ is not compact for any non-trivial simple factor $G'$ of $G$ and $H$ is connected reductive without non-trivial $F$-characters. Then
$$ N({\bf X}, T) \sim  \sum_{\xi\in (Br(X)/Br(F))} (\prod_{v<\infty_F}  N_v({\bf X}, \xi)) \cdot N_{\infty_F} (X, T, \xi)$$
as $T\rightarrow \infty$ under the assumption (\ref{equi-dis}), where $\bf X$ is a separated scheme of finite type over $o_F$ such that ${\bf X}\times_{o_F} F=X$.
\end{thm}
\begin{proof} Decompose
$${\bf X}(o_F) =  \bigcup_{i}  ({\bf X}(o_F)\cap x_i St({\bf X}))$$ with $x_i\in {\bf X}(o_F)$ and $$ {\bf X}(o_F)\cap x_i St({\bf X})  = \bigcup_{j} (y_j^{(i)} G(F) \cap x_i St({\bf X}))= \bigcup_{j} (y_j^{(i)} G(F) \cap y_j^{(i)} St({\bf X})) $$ with $y_j^{(i)} \in {\bf X}(o_F)$ and
$$y_j^{(i)} G(F) \cap x_i St({\bf X})= \bigcup_{k} z_k^{(i,j)} \Gamma $$ with $z_k^{(i,j)}\in {\bf X}(o_F)$. Therefore $$ N({\bf X}, T) \sim \sum_i \sum_j \sum_k \frac{\nu^{(i,j,k)}_{\infty_F}( \Gamma_{i,j,k}\backslash H_{i,j,k}(F_\infty)) }{\lambda_{\infty_F} (\Gamma \backslash G(F_\infty)) } d_{\infty_F} (z_k^{(i,j)} \cdot G(F_{\infty}) \cap X(F_\infty,  T)) $$ as $T\rightarrow \infty$ by the assumption (\ref{equi-dis}), where $H_{i,j,k}$ is the stabilizer of $z_k^{(i,j)}$ in $G$, $\Gamma_{i,j,k}= H_{i,j,k}(F) \cap \Gamma$ and $\nu^{(i,j,k)}$ is the Tamagawa measure of $H_{i,j,k}$.

Since there is $g_k\in G(F)$ such that $z_{k}^{(i,j)}=y_{j}^{(i)}\cdot g_k$, one has the isomorphism between $H_{i,j,k}$ and the stabilizer $H_{i,j}$ of $y_j^{(i)}$ over $F$ given by the conjugation with $g_k$. This implies that
$$ \nu^{(i,j,k)}_{\infty_F}( \Gamma_{i,j,k}\backslash H_{i,j,k}(F_\infty))= \nu^{(i,j)}_{\infty_F}( g_k\Gamma g_k^{-1} \cap H_{i,j}(F) \backslash H_{i,j}(F_\infty)) $$
where $\nu^{(i,j)}$ is the Tamagawa measure over $H_{i,j}$.

Since there is $s_k\in St({\bf X})$ such that $z_{k}^{(i,j)}=y_j^{(i)} \cdot s_k$, one obtains
$$ \aligned & \nu^{(i,j)}_{\infty_F}( g_k\Gamma g_k^{-1} \cap H_{i,j}(F) \backslash H_{i,j}(F_\infty)) \\
= &
  \nu^{(i,j)} (H_{i,j}(F)\backslash H_{i,j}(F) (H_{i,j}(\Bbb A_F) \cap g_k St({\bf X}) g_k^{-1})) \\
 & \cdot \prod_{v<\infty_F} \nu^{(i,j)}_v (H_{i,j}(F_v) \cap g_k St({\bf X}(o_{F_v})) g_k^{-1})^{-1} \\
 = &
  \nu^{(i,j)} (H_{i,j}(F)\backslash H_{i,j}(F) h_k (H_{i,j}(\Bbb A_F) \cap St({\bf X}))) \\
 & \cdot \prod_{v<\infty_F} \nu^{(i,j)}_v (H_{i,j}(F_v) \cap  St({\bf X}(o_{F_v})))^{-1}
\endaligned
$$ where $h_k=g_k\cdot s_k^{-1} \in H_{i,j}(\Bbb A_F)$. Since $G(\A_F)=G(F)St(\bf X)$ by the strong approximation property of $G$, we have $$\bigcup_{k}H_{i,j}(F) h_k (H_{i,j}(\Bbb A_F) \cap St({\bf X}))=H_{i,j}(\Bbb A_F)$$ by Proposition \ref{refine}. Therefore
$$ \aligned N({\bf X}, T) \sim & \sum_i \frac{d_{\infty_F} (x_i \cdot G(F_{\infty}) \cap X(F_\infty,  T)) }{\lambda_{\infty_F} (\Gamma \backslash G(F_\infty)) }  \\
& ( \sum_j  \tau(H_{i,j}) \prod_{v<\infty_F} \nu^{(i,j)}_v (H_{i,j}(F_v) \cap  St({\bf X}(o_{F_v})))^{-1} )
\endaligned $$ as $T\rightarrow \infty$, where $\tau(H_{i,j})$ is the Tamagawa number of $H_{i,j}$.

Since
 $$\nu^{(i,j)}_v (H_{i,j}(F_v) \cap  St({\bf X}(o_{F_v})))^{-1}=\frac{ d_v(y_j^{(i)} St({\bf X}(o_{F_v})))}{\lambda_v (St({\bf X}(o_{F_v})))}=\frac{ d_v(x_i St({\bf X}(o_{F_v}))}{\lambda_v (St({\bf X}(o_{F_v})))}$$ and $H_{i,j}$ is an inner form of $H$, one concludes that $\tau(H_{i,j})=\tau(H)$ by (5.1.1) in \cite{kot} and
 $$  N({\bf X}, T) \sim  \sharp Pic(H) \sum_i \frac{d_{\infty_F} (x_i \cdot G(F_{\infty}) \cap X(F_\infty,  T)) \prod_{v<\infty_F} d_v(x_i St({\bf X}(o_{F_v}))}{\lambda_{\infty_F} (\Gamma \backslash G(F_\infty))\prod_{v<\infty_F} \lambda_v (St({\bf X}(o_{F_v})))} $$ as $T\rightarrow \infty$ by Proposition \ref{sha} and Theorem 10.1 in \cite{San}.

 By the strong approximation property for $G$ and \cite{kot88} and Proposition 2.10 (ii) in \cite{CTX}, one has
 $$ \aligned  & \lambda_{\infty_F} (\Gamma \backslash G(F_\infty))\prod_{v<\infty_F} \lambda_v (St({\bf X}(o_{F_v})))\\
 = & \lambda_G(G(F)\backslash G(F) St({\bf X}))= \lambda_G (G(F)\backslash G(\Bbb A_F))=1 \endaligned $$ and $\sharp Pic(H)= \sharp (Br(X)/Br(F))$. Then
 $$  N({\bf X}, T) \sim \sharp (Br(X)/Br(F)) \int_{(\prod_{v<\infty_F} {\bf X}(o_{F_v}) \times X(F_\infty, T))^{Br(X)}} d_{X} $$ as $T\rightarrow \infty$ by Proposition \ref{equiv}, 1). The result follows from Lemma \ref{average}.
\end{proof}

For general reductive groups, one has the similar result with some restriction.
\begin{thm} \label{thm:red} Suppose $G$ and $H$ are connected reductive groups without non-trivial characters over $F$ and $X=H\backslash G$. If $G'(F_\infty)$ is not compact for any non-trivial simple factor $G'$ of the semi-simple part of $G$ and the map $G(F)\xrightarrow{p} X(F)$ in (\ref{induces}) is surjective, then
$$ N({\bf X}, T) \sim \frac{\tau(H)}{\tau(G)} (\sum_{\chi} \prod_{v<\infty_F}  N_v({\bf X}, \overline{\chi}_v)) \cdot d_{\infty_F} (X(F_{\infty}, T))$$
as $T\rightarrow \infty$ under the assumption (\ref{equi-dis}), where $\chi$ runs over all characters in (\ref{chara-brau}) and $\bf X$ is a separated scheme of finite type over $o_F$ such that ${\bf X}\times_{o_F} F=X$.
\end{thm}
\begin{proof} By Theorem 7.28 in \cite{Dem}, one has $[G,{\bf X}]= G(F)St({\bf X})$. By Proposition \ref{arith} and the assumption (\ref{equi-dis}), one has $$ N({\bf X}, T) \sim \frac{d_{\infty_F} ( X(F_\infty,  T)) }{\lambda_{\infty_F} (\Gamma \backslash G(F_\infty)) }  r({\bf X} \cdot 1_{\Bbb A})$$ as $T\rightarrow \infty$. Since
$$ \lambda_{\infty_F} (\Gamma \backslash G(F_\infty)) \prod_{v<\infty_F} \lambda_v (St({\bf X}(o_{F_v})) ) = \lambda_G(G(F) \backslash [G, {\bf X}]) =\frac{\tau(G)}{s({\bf X})},  $$ one has
$$ N({\bf X}, T) \sim \frac{\tau(H)}{\tau(G)}(\sum_{\chi} \prod_{v<\infty_F}  N_v({\bf X}, \overline{\chi}_v)) \cdot d_{\infty_F} (X(F_{\infty}, T))$$
as $T\rightarrow \infty$ by Corollary \ref{key} and Lemma \ref{stand}.
\end{proof}

\bigskip

\section{Application to norm equations}\label{sec.norm}

In this section, we will apply Theorem \ref{thm:red} to study the asymptotic formula of the number of integral solutions for the scheme $\bf X$ over $o_F$ whose generic fiber $X={\bf X}\times_{o_F} F$ is a trivial torsor of anisotropic torus. Let $G$ be a torus and $U$ be an open subgroup of $G(\Bbb A_F)$. A character
\begin{equation} \label{chara-ani} \varrho: \ \ \ G(\Bbb A_F)/G(F)U \longrightarrow \Bbb C^\times \end{equation} can be written as $\varrho=\prod_{v\in \Omega_F} \varrho_{F_v}$
where  $\varrho_{F_v}$ is induced by
$$ \varrho_{F_v}: G(F_v) \rightarrow  G(\Bbb A_F)/G(F)U \xrightarrow{\varrho} \Bbb C^\times $$ for any $v\in \Omega_F$. Let $X$ be a trivial torsor with a fixed rational point $P\in X(F)$ and $p: G\rightarrow X$ be a morphism induced by $P$. Then each $\varrho_{F_v}$ also induces a continuous function $\overline{\varrho}_v$ on $X(F_v)$ by setting
$$\overline{\varrho}_v (x_v) = \varrho_{F_v} (p^{-1}(x_v))$$
for any $x_v\in X(F_v)$ with $v\in \Omega_F$.  One can reformulate Theorem \ref{thm:red} in more flexible way.

\begin{lem} \label{lem:change-st} Let $G$ is an anisotropic torus and $X$ is a trivial torsor of $G$. If $\bf X$ is a separated scheme of finite type over $o_F$ such that ${\bf X}\times_{o_F} F=X$ and $U$ is an open subgroup of $St({\bf X})$,  then
$$N({\bf X}, T) \sim \frac{1}{\tau(G)}\sum_{\varrho} (\prod_{v<\infty_F}  N_v({\bf X}, \overline{\varrho}_v)) N_{\infty} (X,T,\overline{\varrho}_\infty) $$ as $T\rightarrow \infty$, where $\tau(G)$ is the Tamagawa number of $G$ and
$$N_v ({\bf X} ,  \overline{\varrho}_v)= \int_{{\bf X}(o_{F_v})} \overline{\varrho}_v
d_v  $$ for any $v< \infty_F$ and
$$ N_{\infty}(X, T, \overline{\varrho}_\infty) = \int_{X(F_\infty, T)} \overline{\varrho}_\infty d_{\infty_F} $$ for $T\gg 0$ and $\varrho$ runs over all characters in (\ref{chara-ani}).
\end{lem}
\begin{proof}
 Decompose $$G(\Bbb A_F)= \bigcup_{i=1}^n a_i G(F) St({\bf X})  \ \ \ \text{and} \ \ \ G(F) St({\bf X})= \bigcup_{j=1}^m b_j G(F)U $$ where we choose $b_j\in St({\bf X})$ for $1\leq j\leq m$.  Then
 $$(\prod_{v<\infty_F}  N_v({\bf X}, \overline{\varrho}_v)) N_{\infty} (X,T,\overline{\varrho}_\infty) = \sum_{i=1}^n\varrho(a_i)\sum_{j=1}^m \varrho(b_j) \int_{D_{i,j}}d_X $$
 where $$  D_{i,j} = (X(F_\infty, T) \times \prod_{v<\infty_F} {\bf X}(o_{F_v})) \bigcap p(a_i b_j G(F)U)$$ for $1\leq i\leq n$ and $1\leq j\leq m$. Since $b_j\in St({\bf X})$, one has that
 $$\int_{D_{i,1}}d_X = \cdots = \int_{D_{i,m}} d_X  $$ for $1\leq i\leq n$. If the restriction of $\varrho$ to $G(F)St({\bf X})/G(F)U$ is not trivial, then $$\sum_{j=1}^m \varrho(b_j)=0 . $$
 Therefore
$$  \sum_{\varrho} (\prod_{v<\infty_F}  N_v({\bf X}, \overline{\varrho}_v)) N_{\infty} (X,T,\overline{\varrho}_\infty)  = \sum_{\chi} \prod_{v<\infty_F} N_v({\bf X}, \overline{\chi}_v) \cdot d_{\infty_F} (X(F_{\infty}, T)) $$
where $\chi$ runs over all characters in (\ref{chara-brau}) and the result follows from Theorem \ref{thm:red}.
\end{proof}

In practice, one can use the idea in \cite{WX} and \cite{WX1} to embed a torus to certain standard torus $\prod_\lambda R_{E_\lambda/F}(\G_m)$ and apply the results in \cite{WX} and \cite{WX1} to get more computable formulae. We explain this point by considering the more concrete scheme $\bf X$ over $o_F$ defined by
\begin{equation} \label{normequ} N_{K/F}(\alpha_1 x_1+\cdots+ \alpha_n x_n)=m \end{equation} where $K/F$ is a finite extension of degree $n$ and $\alpha_1,\cdots, \alpha_n\in o_K$ which are $F$-linear independent and $m\in o_F\setminus\{0\}$. Then $X={\bf X}\times_{o_F} F$ is a torsor of the norm one torus $G=R^{1}_{K/F}(\G_m)$.
Let $$n_\infty=\sum_{v\in \infty_F}(n_v-1)$$ where $n_v$ is the number of places of $K$ above $v$. If ${\bf X}(o_F) \neq \emptyset$ and $n_\infty > 0$, then there exists a constant $c>0$ such that
$$N({\bf X}, T) \sim c(\log T)^{n_\infty}$$
as $T\rightarrow \infty$ by the proof of generalized Dirichlet Unit Theorem (see Theorem 5.12 of Chapter 5 in \cite{PR94}).  We will explain how to determine the above constant $c$ explicitly.

\medskip

Let $\{\sigma_1,\cdots,\sigma_n\}$ be the set of all embedding of $K$ over $F$ and $\Delta=det(\sigma_i(\alpha_j))_{n\times n}$.

\begin{lem} \label{lem:d-form} Let $$ u= \sum_{j=1}^n \alpha_j x_j \ \ \ \text{and} \ \ \ \sigma_i(u)= \sum_{j=1}^n \sigma_i(\alpha_j) x_j  $$ for $1\leq i\leq n$.  Then $$\omega=m^{-1}\Delta^{-1}\sigma_1(u) \cdot d\sigma_2(u)\wedge d\sigma_3(u)\wedge\cdots \wedge d\sigma_n(u)$$ is an invariant differential form of $X$.
\end{lem}
\begin{proof} Since $$\omega=\Delta^{-1} \cdot \frac{d\sigma_2(u)}{\sigma_2(u)}\wedge \cdots \wedge \frac{d\sigma_n(u)}{\sigma_n(u)}, $$ one has that $\omega$ is an invariant differential form of $X\times_F K$.

Let $M_{j}$ be the minor for entry $\sigma_1(\alpha_{j})$ in $(\sigma_i(\alpha_j))_{n\times n}$. Then we have
\begin{equation}\label{eq:ome} d\sigma_2(u)\wedge\cdots \wedge d\sigma_n(u)=\sum_{j=1}^{n}M_j dx_1\wedge\cdots \wedge \widehat{dx_j} \wedge\cdots \wedge dx_n.\end{equation}
 Since $$ \sum_{i=1}^n \frac{d\sigma_i(u)}{\sigma_i(u)}=0 \ \  \Rightarrow  \ \ \sum_{i=1}^n \tr_{K/F}(\alpha_i/u)\cdot dx_i=0, $$ one obtains that  \begin{equation}\label{eq:dx} dx_1=-\tr_{K/F}(\alpha_1/u)^{-1}(\sum_{i=2}^n \tr_{K/F}(\alpha_i/u)\cdot dx_i).\end{equation}
Replacing  $dx_1$ in (\ref{eq:ome}) with (\ref{eq:dx}),  we have
$$\aligned & d\sigma_2(u)\wedge\cdots \wedge d\sigma_n(u) \\
= & \tr_{K/F}(\alpha_1/u)^{-1} (\sum_{j=1}^{n}(-1)^{1+j}M_j \tr_{K/F}(\alpha_j/u)) \cdot dx_2\wedge\cdots \wedge dx_n.
\endaligned $$
 Since $$\sum_{j=1}^{n}(-1)^{1+j}M_j\cdot \sigma_1(\alpha_j)=\Delta \ \ \ \text{and} \ \ \ \sum_{j=1}^{n}(-1)^{1+j}M_j\cdot \sigma_i(\alpha_j)=0$$ for $2\leq i \leq n$, one obtains
  $$\aligned &  \tr_{K/F}(\alpha_1/u)^{-1} \sum_{j=1}^{n}(-1)^{1+j}M_j\tr_{K/F}(\alpha_j/u)\\
= & \tr_{K/F}(\alpha_1/u)^{-1}(\sum_{j=1}^{n}(-1)^{1+j}M_j\cdot\sigma_1(\alpha_j))\sigma_1(u)^{-1} \\
& +\sum_{i=2}^n (\sum_{j=1}^{n}(-1)^{1+j}M_j\cdot \sigma_i(\alpha_j))\sigma_i(u)^{-1} \\
= & \tr_{K/F}(\alpha_1/u)^{-1}  \sigma_1(u)^{-1} \Delta.
\endaligned $$
Hence $$d\sigma_2(u)\wedge\cdots \wedge d\sigma_n(u)=\tr_{K/F}(\alpha_1/u)^{-1}\sigma_1(u)^{-1} \Delta \cdot dx_2\wedge\cdots \wedge dx_n$$ and $$\omega=m^{-1} \tr_{K/F}(\alpha_1/u)^{-1}\cdot dx_2\wedge\cdots \wedge dx_n . $$ This implies that $\omega$ is a differential form of $X$.
\end{proof}

We will use this volume form $\omega$ to calculate $d_{\infty_F} (X(F_{\infty}, T))$.

\begin{lem}\label{lem:vol}(1) Suppose $v$ is a real place.

If there are $r\geq 1$ real places and $s$ complex places of $K$ over $v$, then $$d_v (X(F_v, T))\sim 2^{r-1}(2\pi)^s\frac{n^{r+s-1}}{(r+s-1)!}|N_{K/F}(\Delta)|_v^{-1}\cdot (\log T)^{r+s-1} $$ as $T\rightarrow \infty$.

If all places over $v$ are complex, then
$$d_v (X(F_v, T))\sim (2\pi)^{s-1}\frac{n^{s-1}}{(s-1)!}|N_{K/F}(\Delta)|_v^{-1}\cdot (\log T)^{s-1}$$ as $T\rightarrow \infty$ with $s=\frac{n}{2}$.

(2) Suppose $v$ be a complex place. Then $$d_v (X(F_v, T))\sim (2\pi)^{n-1}\frac{n^{n-1}}{(n-1)!}|N_{K/F}(\Delta)|_v^{-1} \cdot (\log T)^{n-1} $$ as $T\rightarrow \infty$.
\end{lem}
\begin{proof} We only prove the case that $v$ be a real place and $r>0$. The rest of cases follows from the exact the same arguments. Assume $\sigma_1,\cdots,\sigma_r$ are real embedding and $\sigma_{r+1},\bar \sigma_{r+1},\cdots,\sigma_{r+s},\bar \sigma_{r+s}$ are complex embedding. Let $B=[-1,1]^n\subset \R^n$. By changing  coordinate $z_i=\sigma_i(u)$ for $1\leq i \leq r+s$ with $\R^n\rightarrow \R^r\times \C^s$, one has
 $$d_v (X(F_v, T))=\int_{T\cdot B\cap X(F_v)} |\omega|
=\int_{T\cdot B'\cap X(F_v)} |\omega'| $$
where $$\omega'=|N_{K/F}(\Delta)|_v^{-1}\frac{dz_2}{|z_2|}\wedge \cdots\wedge \frac{dz_r}{|z_r|}\wedge \frac{dz_{r+1}}{|z_{r+1}|}\wedge \frac{d\bar z_{r+1}}{|\bar z_{r+1}|}\wedge \cdots \wedge \frac{dz_{r+s}}{|z_{r+s}|}\wedge \frac{d\bar z_{r+s}}{|\bar z_{r+s}|}$$
 and $B'$ is the image of $B$ in $\R^r\times \C^s$ under this change coordinate.

For any $\delta >0$, we define $$B(\delta)=\{(x_1,\cdots,x_r,x_{r+1}, \cdots , x_{r+s})\in \R^r\times \C^s: \ |x_i|\leq \delta \ \text{for all $1\leq i\leq r+s$} \}. $$  Since
there exist $\delta_1>\delta_2 >0$ such that $B(\delta_2)\subset B'\subset B(\delta_1)$
 and $$\int_{ B(T\cdot\delta_2)\cap X(F_v)}|\omega'| \leq d_v (X(F_v, T))\leq \int_{ B(T\cdot\delta_1)\cap X(F_v)}|\omega'| ,$$
one only needs to show that
$$\int_{ B(T\cdot\delta)\cap X(F_v)}|\omega'|\sim 2^{r-1}(2\pi)^s\frac{n^{r+s-1}}{(r+s-1)!}|N_{K/F}(\Delta)|_v^{-1}\cdot (\log T)^{r+s-1}$$ as $T\rightarrow \infty$ for any $\delta>0$.

Since $B(T\cdot \delta)\cap X(F_v)$ has $2^{r-1}$ connected components and the integral over each connected component is the same. Therefore
$$\int_{B(T\cdot \delta)\cap X(F_v)}|\omega'|=2^{r-1} \int_{B_{X,T}}|\omega'| $$
where $B_{X,T}$ is the connected component of $ B(T\cdot\delta)\cap X(F_v)$ such that the first $r$-real coordinates are positive.

By using the the polar coordinates $$\begin{cases} z_i = e^{\rho_i} \ \ \ & \text{ for $i=1,\cdots,r$ } \\
z_j= e^{\rho_j+{\bf i} \theta_j} \ \ \ & \text{ with $0\leq \theta_j < 2\pi$ for $j=r+1, \cdots r+s$},  \end{cases}$$
one obtains that
$$\int_{B_{X,T}}|\omega'|= |N_{K/F}(\Delta)|_v^{-1} (2\pi)^s \int_{V(T)} (\bigwedge_{i=2}^r d\rho_i)\wedge (\bigwedge_{j=1}^s d (2\rho_{r+j}))  $$
where
$$V(T)=\{ (\rho_i)\in (-\infty, \ \log (T\cdot\delta)]^{r+s}: \  \sum_{i=1}^r\rho_i+\sum_{j=1}^s 2\rho_{r+j}=\log |m|_v \} . $$
By substituting $\rho_i$ by $\rho_i \log T $ for $1\leq i\leq r+s$, one has
$$\int_{V(T)} (\bigwedge_{i=2}^r d\rho_i)\wedge (\bigwedge_{j=1}^s d (2\rho_{r+j})) \sim (\log T)^{s+r-1} \int_{V} (\bigwedge_{i=2}^r d\rho_i)\wedge (\bigwedge_{j=1}^s d (2\rho_{r+j})) $$
as $T \rightarrow \infty$, where
 $$ \aligned V & =\{(\rho_i)\in (-\infty,1]^{r+s}: \ \sum_{i=1}^r\rho_i+\sum_{j=1}^s 2\rho_{r+j}=0 \}\\
 & = \{(u_i)\in (-\infty, 0]^{r+s}: \ \sum_{i=1}^r u_i + \sum_{j=1}^s 2u_{r+j}= -n \} \\
  & = \{ (u_2, \cdots, u_{r+s})\in (-\infty, 0]^{r+s-1}: \sum_{i=2}^r u_i+\sum_{j=1}^s 2u_{r+j} \geq -n \}  \endaligned $$ with $u_i=\rho_i-1$ for $i=1, \cdots, r+s$. Therefore
$$ \int_{V} (\bigwedge_{i=2}^r d\rho_i)\wedge (\bigwedge_{j=1}^s d (2\rho_{r+j})) = n^{r+s-1}/(r+s-1)! $$ by the standard computation. The proof is complete.
\end{proof}

 If $F=\Q$ and $\alpha_1,\cdots, \alpha_n\in o_K$ are linear independent over $\Bbb Q$ such that $$L=\Bbb Z \alpha_1+\cdots + \Bbb Z \alpha_n$$ is an order of $K$, one can associate the narrow ring class field $H_L$
corresponding to the order $L$ with the Artin reciprocity isomorphism
 \begin{equation} \label{artin} \psi_{H_L/K}:  \  \Bbb I_K/ K^\times (K_\infty^+ \prod_{p< \infty}
L_p^\times)  \xrightarrow{\cong}  \Gal(H_L/K) \end{equation} where $K_\infty^+$ is the connected component of $1$ inside $K_\infty$,
$L_p$ is the $p$-adic completion of $L$ inside $K_p=K\otimes_\Bbb Q \Bbb Q_p$ and $L_p^\times$ is the unit group of $L_p$ for any prime $p$.

Let $$G=R^1_{K/\Q}(\G_m) \ \ \ \text{ and } \ \ \ {\bf G}(\Z_p)=\{\xi \in L_ p^\times: \ N_{K_p/\Q_p}(\xi)=1\} $$ for any prime $p$.
Set $G(\Bbb R)^+$ to be the connected component of $1$ inside
 $$ G(\Bbb R) = \{ x\in K_\infty^\times: \ N_{K_\infty/\Bbb R}(x)=1 \} . $$
The homomorphism induced by the natural inclusion
\begin{equation} \label{natural} \lambda_K: G(\Bbb A_\Q)/G(\Q)(G(\Bbb R)^+ \prod_{p< \infty}\bold G(\Z_p))
\longrightarrow \Bbb I_K/ K^\times (K_\infty^+ \prod_{p< \infty}
L_p^\times) \end{equation}  is well-defined. Moreover, if $x_\Bbb A\in ker(\lambda_K)$, then there are $$\alpha\in K^\times \ \ \text{ and } \ \ y_\Bbb A\in K_\infty^+ \prod_{p< \infty}
L_p^\times $$ such that $x_\Bbb A=\alpha \cdot y_\Bbb A$. Therefore
$$ N_{K/\Bbb Q}(\alpha^{-1})=N_{K/\Bbb Q}(\alpha^{-1} x_\Bbb A)=N_{K/\Bbb Q}(y_\Bbb A)\in \Bbb Q\cap (\Bbb R^+\times \prod_{p} \Bbb Z_p^\times )=\{1\} $$ where $\Bbb R^+$ is the set of positive reals. This implies that $\lambda_K$ is injective.

Consider the Diophantine equation (\ref{normequ}) $\bf X$ over $\Z$. Then ${\bf X}(\Z)\neq \emptyset$ if and only if there is
$$(x_{1,p},\cdots,x_{n,p})_{p\leq \infty}\in \prod_{p\leq \infty} \bold X(\Bbb
Z_p) $$ such that $$\psi_{H_L/K}(
( \sum_{i=1}^n  \alpha_i x_{i,p})_{p\leq \infty})=1 $$ where
$$\sum_{i=1}^n  \alpha_i x_{i,p}\in K_p=K\otimes_{\Q} \Q_p $$
by Corollary 1.6 in \cite{WX}.

\begin{lem} \label{lem:ext2} Let $K/\Q$ be a finite extension of degree $n$ and $\bf X$ be the scheme defined by the norm equation
$$ N_{K/\Q}(\alpha_1 x_1+\cdots+\alpha_n x_n)=m $$
with $\alpha_1, \cdots, \alpha_n \in o_K$ such that $L=\Z \alpha_1+\cdots + \Z \alpha_n$ is an order of $K$ and $m\in \Z$ with $m\neq 0$. If $X={\bf X}\times_{\Z} \Q$ is a trivial torsor of $G=R_{K/\Q}^1(\G_m)$, then
$$ N({\bf X}, T) \sim \frac{h_G}{h_L \cdot \tau(G)}\sum_{\phi} \prod_{p \ \text{primes}}  N_p({\bf X}, \phi) N_\infty(X, T, \phi)$$ as $T\rightarrow \infty$, runs over all characters of $\Gal(H_L/K)$ and
$$h_G=[G(\Bbb A_\Q):G(\Q)(G(\Bbb R)^+ \prod_{p< \infty}\bold G(\Z_p))] \ \ \ \text{and} \ \ \ h_L=[\Bbb I_K: K^\times (K_\infty^+ \prod_{p< \infty}L_p^\times)] $$ with the above notation and $\tau(G)$ is the Tamagawa number of $G$ and
$$ N_p({\bf X}, \phi)= \int_{{\bf X}(\Z_p)} \phi( \psi_{H_L/K}(\sum_{i=1}^n \alpha_i x_{i,p})) d_p(x_{1,p},\cdots, x_{n,p}) $$ for all primes $p$ and
$$ N_\infty(X,T, \phi)= \int_{X(\R, T)} \phi( \psi_{H_L/K}(\sum_{i=1}^n \alpha_i x_{i,\infty})) d_\infty(x_{1,\infty},\cdots, x_{n,\infty}) $$ with the Artin map $\psi_{H_L/K}$ in (\ref{artin}).
\end{lem}

\begin{proof} Fix $(\varsigma_1, \cdots, \varsigma_n)\in X(\Q)$. One can apply Lemma \ref{lem:change-st} by using
$$\overline{\varrho}_p (x_{1,p}, \cdots, x_{n,p}) = \varrho_{F_v} ((\sum_{i=1}^n \alpha_i x_{i,p})(\sum_{i=1}^n \alpha_i \varsigma_i)^{-1})$$
for any $(x_{1,p}, \cdots, x_{n,p})\in X(\Q_p)$ with all primes $p\leq \infty$.

Since the natural inclusion $\lambda_K$ in (\ref{natural}) is injective, one concludes that
$$ N({\bf X}, T) \sim \frac{h_G}{h_L \cdot \tau(G)}\sum_{\phi} \prod_{p\leq \infty}  N_p({\bf X}, \phi) N_\infty(X, T, \phi)$$ as $T\rightarrow \infty$, where $\phi$ runs over all characters of $\Gal(H_L/K)$  by the Artin reciprocity law (\ref{artin}). \end{proof}

We further specialize to the case that $K/\Q$ is a real quadratic extension with the discriminant $D$ and $L=\Bbb Z \alpha_1+\Bbb Z \alpha_2$ is the maximal order $o_K$ of $K$. Let $P(D)$ be the set of all prime divisors of $D$ and $$ m=(-1)^s p_1^{e_1}\cdots p_f^{e_f}\prod_{q_i\in P(D)}q_i^{t_i}$$  where $p_1,\cdots,p_f$ are the distinct primes and prime to $D$. Let $$Q_1=\{p_i: \ (\frac{D}{p_i})=1  \ \text{with} \ 1\leq i\leq f\}$$ and $$ Q_2=\{p_i: \
(\frac{D}{p_i})=-1  \ \text{with} \ 1\leq i\leq f \}. $$

Since $\Gal(K/\Bbb Q)$ acts on the two real places transitively, one real place is ramified in $H_L/K$ if and only if  the other real place is ramified in $H_L/K$.  Since the global element $-1$ in $K$ gives identity in $\Gal(H_L/K)$ via the Artin map, one obtains that the complex conjugations over these two real places are the same by using the product of the local Artin maps when a real place $v$ of $K$ is ramified in $H_L/K$. Let
$\sigma_{-1}$ be the complex conjugation if a real place of $K$ is ramified in $H_L/K$; otherwise $\sigma_{-1}$ is the identity in $\Gal(H_L/K)$.

If $p_i\in Q_1$, then $p_i$ splits into two primes $\frak p_1$ and $\frak p_2$ in $K$. Since the global element $p_i$ gives identity in $\Gal(H_L/K)$ via the Artin map, one obtains that the Frobenius of $\frak p_1$ and the Frobenius of $\frak p_2$ are inverse to each other in $\Gal(H_L/K)$ by using the product of the local Artin maps. Let $\sigma_{p_i}$ be the Frobenius of $\frak p_1$ in $\Gal(H_L/K)$. Then the Frobenius of $\frak p_2$ in $\Gal(H_L/K)$ is $\sigma_{p_i}^{-1}$. Define
$$\delta_{p_i}(\phi)=\begin{cases} (e_i+1)\phi (\sigma_{p_i})^{e_i} \ \ \
& \text{if }\phi (\sigma_{p_i}^2)=1 \\
(1-\phi(\sigma_{p_i})^{2e_i+2})(\phi(\sigma_{p_i})^{e_i}-\phi(\sigma_{p_i})^{e_i+2
})^{-1}  \ \ \ & \text{otherwise.}
\end{cases}$$ for any character $\phi$ of $\Gal(H_L/K)$. Then $\delta_{p_i}(\phi)$ is independent of the choice of the place $\frak p_1$ of $K$.

If $p_i\in Q_2$, then $p_i$ is inert in $K$. Then the Frobenius $\sigma_{p_i}$ of $p_i$ satisfies $$\sigma_{p_i}=\psi_{H_L/K,p_i}(p_i)=\prod_{v\neq p_i} \psi_{H_L/K,v}(p_i)^{-1}$$ by the reciprocity law. Since $H_L/K$ is unramified at all finite places, we have $$\sigma_{p_i}=\prod_{v\mid \infty} \psi_{H_L/K,v}(p_i)^{-1}\cdot \prod_{v< \infty, v\neq p_i} \psi_{H_L/K,v}(p_i)^{-1}=1\cdot 1=1.$$

If $p_i\in P(D)$, the Frobenius of the unique place of $K$ over $p_i$ in $\Gal(H_L/K)$ is denoted by $\sigma_{p_i}$.

\begin{prop} \label{prop:pell-n}
 Suppose $K/\Q$ is a real quadratic extension and $L=\Bbb Z \alpha_1+\Bbb Z \alpha_2$ is the maximal order $o_K$ of $K$. With the above notation, we define
 $$c_m= \sum_{\phi}\phi(\sigma_{sgn(m)})\cdot \prod_{p_i\in P(D)}\phi(\sigma_{p_i})^{t_i}\cdot \prod_{p_i\in Q_1}\delta_{p_i}(\phi)  $$
 where $\phi$ runs over all characters of $\Gal(H_L/K)$ and
 $$\sigma_{sgn(m)}= \begin{cases} \sigma_{-1} \ \ \ & \text{if $m<0$} \\
 1 \ \ \ & \text{if $m>0$. } \end{cases} $$ Then the norm equation $$N_{K/\Q}(\alpha_1 x+ \alpha_2 y)=m$$ is solvable over $\Z$ if and only if this equation is solvable over $\Bbb Z_p$ for all primes $p\leq \infty$ and $c_m\neq 0$.
  Moreover
 $$N({\bf X}, T) \sim \frac{2c_m}{h_K^+ \sqrt{D}} \cdot \frac{\log T}{\log\epsilon} $$
as $T\rightarrow \infty$, where $\epsilon$ is the unique minimal unit of $o_K$ such that $\epsilon >1$ and $N_{K/\Q}(\epsilon)=1$ and $$h_K^+=[\Bbb I_K:K^\times (K_\infty^+ \prod_{v<\infty_K} o_{K_v}^\times )]$$ with the connected component $K_\infty^+$ of 1 in $K_\infty^\times$ is the narrow class number of $K$.
\end{prop}
\begin{proof}
Let $\bf X$ be the scheme over $\Bbb Z$ defined by the above norm equation. Since
$$\prod_{p \leq \infty} {\bf X}(\Z_p) \neq \emptyset,  $$ one gets that $X={\bf X}\times_\Z \Q$ is a trivial torsor of $G=R_{K/\Q}^1(\G_m)$. In order to apply Lemma \ref{lem:ext2}, one needs to compute the integrals in each term.

If $(p,m)=1$, then $\psi_{H_L/K}$ takes the trivial value over ${\bf X}(\Z_p)$ and
$$N_p({\bf X}, \phi)= \int_{{\bf X}(\Z_p)}  d_p=\int_{{\bf G}(\Z_p)} d_p $$ for all character $\phi$ of $\Gal(H_L/K)$.

If $p_i\in P(D)$, then $$ \psi_{H_L/K}(\alpha_1 x_{p_i}+\alpha_2 y_{p_i})=\sigma_{p_i}^{t_i}$$ for all $(x_{p_i}, y_{p_i})\in {\bf X}(\Z_{p_i})$. Therefore
$$N_{p_i}({\bf X}, \phi)=  \phi(\sigma_{p_i})^{t_i}\int_{{\bf G}(\Z_{p_i})} d_{p_i}$$ for all character $\phi$ of $\Gal(H_L/K)$.

If $p_i\in Q_1$, then $p$ splits into two places $\frak p_1$ and $\frak p_2$ in $K/\Q$. Let $\sigma_{p_i}$ be the Frobenius of $H_L/K$ at $\frak p_1$. Therefore
$$\aligned & N_{p_i}({\bf X}, \phi) = \sum_{i=0}^{e_i} \int_{z_1z_2=m,v_{p_i}(z_1)=i} \phi(\sigma_{p_i}^i\cdot \sigma_{p_i}^{-(e_i-i)}) d_{p_i}\\
=& \sum_{i=0}^{e_i}\phi(\sigma_{p_i})^{2i-e_i} \int_{{\bf G}(\Z_p)}  d_{p_i}=\delta_{p_i}(\phi)\int_{{\bf G}(\Z_{p_i})}  d_{p_i}
\endaligned$$ for all character $\phi$ of $\Gal(H_L/K)$.

If $p=\infty$ and a real place over $\infty$ is ramified in $H_L/K$, then $$\psi_{H_L/K}(\alpha_1 x_\infty+\alpha_2 y_\infty)=\begin{cases} \sigma_{-1} \ \ \ & \text{when $\alpha_1 x_\infty+\alpha_2 y_\infty<0$ } \\
1 \ \ \ & \text{otherwise} \end{cases} $$
over such a real place. Therefore
$$N_\infty(X, T, \phi)= \begin{cases} \phi(\sigma_{-1}) d_\infty (X(\R, T)) \ \ \ \ & \text{if $m<0$} \\
 d_\infty (X(\R, T)) \ \ \ \ & \text{if $m>0$} \end{cases} $$
for all character $\phi$ of $\Gal(H_L/K)$.

Let $$h_G=[G(\Bbb A_\Q):G(\Q)(G(\Bbb R)^+ \prod_{p< \infty}\bold G(\Z_p))] $$ where $G(\R)^+$ is the connected component of 1. Then
 $$\tau (G)=h_G\prod_{p<\infty}\int_{{\bf G}(\Z_p)}  d_p  \cdot \int_{{\bf G}(\Z)\backslash G(\R)}d_\infty $$ where $\tau(G)$ is the Tamagawa number of $G$ and $${\bf G}(\Z)=G(\Q) \cap (G(\Bbb R)^+ \prod_{p< \infty}\bold G(\Z_p)) . $$
By using the invariant differential in Lemma \ref{lem:d-form}, one obtains that
 $$\int_{{\bf G}(\Z)\backslash G(\R)}d_\infty= \frac{2}{ \sqrt{D}} \log \epsilon $$ where $\epsilon$ is the unique minimal unit of $o_K$ satisfying with $\epsilon >1$ and $N_{K/\Q}(\epsilon)=1$. By Lemma \ref{lem:vol}, we have $$d_\infty (X(\R, T))\sim \frac{4}{D}\cdot \log T$$ as $T\rightarrow \infty$. The result follows from combining the above computation and Lemma \ref{lem:ext2} and Proposition \ref{arith}. \end{proof}

 By Proposition \ref{prop:pell-n}, the negative Pell equation $$x^2-\delta y^2=-1$$ where $\delta$ is a square-free positive integer with $\delta \not\equiv 1 \mod 4$ is solvable over $\Z$ if and only if $\sigma_{-1}$ is trivial in $\Gal(H_L/K)$, which is equivalent to that the narrow class number $h_K^+$ of $K$ is equal to the the class number $h_K$ of $K$. This is the well-known classical result. In this case, one has
$$N({\bf X}, T) \sim \delta^{-\frac{1}{2}} \frac{\log T}{\log \epsilon}$$ as $T\rightarrow \infty$ by Proposition \ref{prop:pell-n}, where
 $\epsilon=x_0+y_0\sqrt{\delta}$ with the integral solution $(x_0,y_0)$ of $x^2-\delta y^2=1$ such that $\epsilon >1$ and $\epsilon$ is minimal.

\medskip

Now we provide a more explicit example. For any integer $m$, one can write $$m=(-1)^{s_0}
2^{s_1}17^{s_2}{p_1}^{e_1}\cdots {p_g}^{e_g} \ \ \ \text{and} \ \ \ \Pi(m)=\{p_1,
\cdots, p_g \} . $$ Decompose $\Pi (m)$ into the disjoint union of the
following subsets
$$\aligned & \Pi_1=\{p\in \Pi (m) : \ (\frac{2}{p})=(\frac{17}{p})=-1  \} \ \ \text{and} \ \ \Pi_2=\{p\in \Pi(m): \
(\frac{34}{p})=-1 \}  \cr
 & \Pi_3= \{p\in \Pi (m): \
(\frac{2}{p})=(\frac{17}{p})=1 \text{ and } (\frac{-7+4\sqrt{2}}{p})=1\}\\
& \Pi_4= \{p\in \Pi (m): \
(\frac{2}{p})=(\frac{17}{p})=1 \text{ and } (\frac{-7+4\sqrt{2}}{p})=-1\}. \endaligned $$ Let
$$m_1=(-1)^{s_0} \prod_{p_i\in \Pi_1} p_i^{e_i} .$$

\begin{exa} \label{exa:pell-34} With the above notation, the equation $x^2-34y^2=m$ is solvable over $\Bbb Z$ if and
only if  $m_1\equiv \pm 1 \mod 8$, $(\frac{m_1}{17})=1$, $(\frac{34}{p_i})=1$ for odd $e_i$ and
\begin{enumerate}

    \item   there is an odd $e_i$ for some $p_i \in \Pi_1$. In this case
 $$N({\bf X}, T) \sim \frac{1}{2\sqrt{34}}\prod_{p_i\in \Pi(m)\setminus \Pi_2}(1+e_i)\cdot \frac{\log T}{\log (35+6\sqrt{34})}$$
 as $T\rightarrow \infty$.

    \item $\Pi_1\neq \emptyset$ and all $e_i$ are even for $p_i\in \Pi_1$. In this case
   $$N({\bf X}, T) \sim \frac{r}{2\sqrt{34}} \prod_{p_i\in \Pi_3\cup \Pi_4}(1+e_i)\cdot \frac{\log T}{\log (35+6\sqrt{34})}$$
as $T\rightarrow \infty$, where
$$ r= (-1)^{s_0+s_2+\frac{1}{2}\sum_{p_i\in \Pi_1}e_i+\sum_{p_i\in \Pi_4}e_i}+ \prod_{p_i\in \Pi_1}(1+e_i) .$$

    \item   $\Pi_1=\emptyset$ and $\sum_{p_i\in \Pi_4} e_i \equiv
    s_0+s_2 \mod 2 $. In this case
    $$N({\bf X}, T) \sim \frac{1}{\sqrt{34}}\prod_{p_i\in \Pi_3\cup \Pi_4}(1+e_i)\cdot \frac{\log T}{\log (35+6\sqrt{34})}$$ as $T\rightarrow \infty$.
\end{enumerate}
\end{exa}

\begin{proof} The narrow Hilbert class field $H_L$ of $K=\Q(\sqrt{34})$ is given by
$$H_L=K(\sqrt{-7+4\sqrt{2}}) \ \ \ \text{ with } \ \ \ \Gal(H_L/K)\cong \mu_4=\{\pm 1,\pm {\bf i}\}. $$
By the notation in Proposition \ref{prop:pell-n}, we have
$$\sigma_{p}= \begin{cases} 1 \ \ \ & \text{if } p\in \{2\}\cup \Pi_2\cup \Pi_3 \cr
-1 \ \ \ & \text{if $p=\{-1,17\}\cup \Pi_4$}\cr
\pm {\bf i} \ \ \ & \text{if
$p\in \Pi_1$}
\end{cases}$$
and $\epsilon=35+6\sqrt{34}$. Let $\phi$ be the generator of the character group of $\Gal(H_L/K)$.

If $\varphi=1 \text{ or }\phi^2$, then $$\delta_{p}(\varphi)=(1+e_i)\varphi(\sigma_{p_i})^{e_i}$$
for $p_i\in \Pi_1\cup \Pi_3\cup \Pi_4$.

If $\varphi=\phi \text{ or }\phi^3$, then $$\delta_{p}(\varphi)=\begin{cases} \frac{1}{2} ({\bf i}^{e_i}+(-{\bf i})^{e_i}) \ \ \ & \text{if  $p_i\in \Pi_1$} \\
(1+e_i)\varphi(\sigma_{p_i})^{e_i} \ \ \ & \text{if
$p_i\in \Pi_3\cup \Pi_4$}.
\end{cases}$$

Therefore $$\aligned c_m & =\sum_{\phi}\phi(\sigma_{-1})\prod_{p_i\in P(34)}\phi(\sigma_{p_i})^{e_i}\prod_{p_i\in \Pi(m)\setminus \Pi_2}\delta_{p_i}(\phi)\\
&=\prod_{p_i\in \Pi(m)\setminus \Pi_2}(1+e_i)+\prod_{p_i\in \Pi_1}(-1)^{e_i}(1+e_i)\cdot \prod_{p_i\in \Pi_3\cup \Pi_4}(1+e_i) \\
& +2(-1)^{s_0+s_2}\prod_{p_i\in \Pi_1}\frac{1}{2}({\bf i}^{e_i}+{(-\bf i)}^{e_i})\cdot \prod_{p_i\in \Pi_3}(1+e_i)\cdot\prod_{p_i\in \Pi_4}(-1)^{e_i}(1+e_i).
\endaligned$$
Since $(\frac{n_1}{17})=1$ by the local solvability condition, one has $$\prod_{p_i\in \Pi_1}(-1)^{e_i}=1 . $$

If there is an odd $e_i$ for some $p_i\in \Pi_1$, then $$c_m=2\prod_{p_i\in \Pi(m)\setminus \Pi_2}(1+e_i).$$

If all $e_i$ are even for all $p_i\in \Pi_1$ or $\Pi_1 =\emptyset$, then $$c_m=2\prod_{p_i\in \Pi(m)\setminus \Pi_2}(1+e_i)+2(-1)^{s_0+s_2+\frac{1}{2}\sum_{p_i\in \Pi_1}e_i + \sum_{p_i\in \Pi_4}e_i}\prod_{p_i\in \Pi_3\cup \Pi_4}(1+e_i).$$ The result follows from Proposition \ref{prop:pell-n} and Example 5.5 in \cite{WX}. \end{proof}

Comparing this example with Example 5.5 in \cite{WX}, one can find that both conditions for existence the integral points look slight different. This is because the $\bf X$-admissible subgroups in \cite[Theorem 1.10 ]{WX} are not unique. In other word, the finite subgroups of $Br(X)$ for testing the existence of the integral points are not unique.

\bigskip

\section{Examples for semi-simple groups} \label{sec.semisimpe}

As application, we will explain that the asymptotic formula of the number of integral solutions in Theorem 1.1 of \cite{EMS} is the same as  the Hardy-Littlewood expectation in sense of \cite{BR95} although the Brauer-Manin obstruction is not trivial in general.

\begin{exa} Let $p(\lambda)$ be an irreducible monic polynomial of degree $n\geq 2$ over $\Bbb Z$ and $\bf X$ be the scheme over $\Bbb Z$ defined by the following equations with variables $x_{i,j}$ for $1\leq i, j\leq n$
$$ det(\lambda I_n-(x_{i,j}))=p(\lambda) $$ and $X=\bf X\times_{\Bbb Z} \Bbb Q$. Then
$$N({\bf X}, T) \sim (\prod_{p<\infty} \int_{{\bf X}(\Bbb Z_p)} d_p)\cdot \int_{{\bf X}(\Bbb R, T)} d_\infty $$
as $T \rightarrow \infty$.
\end{exa}

\begin{proof}  Since both $SL_n$ and $GL_n$ act on $X$ by $ x\circ g = gxg^{-1}$, one has that $X$ is the homogeneous space of $SL_n$ and $GL_n$ with a rational point
$$v= \left(\begin{array} [c]{llll}
0 \ \  0 \ \cdots \ \ 0 \ \ -a_n  \\ 1 \ \ 0 \ \cdots  \ \ 0 \ \ -a_{n-1} \\
 \ \ \ \cdots \ \cdots \ \cdots \\
0 \ \ 0 \ \cdots \ \ 1 \ \ -a_1
\end{array}
\right) \in {X}(\Bbb Q)$$ where $p(\lambda)=\lambda^n+a_1 \lambda^{n-1} + \cdots + a_0$. Moreover, the stabilizers of $v$ inside $SL_n$ and $GL_n$ are $S\cong R_{K/\Bbb Q}^1 (\Bbb G_m)$ and $R_{K/\Bbb Q} (\Bbb G_m)$ respectively, where $K=\Bbb Q(\theta)$ and $\theta$ is a root of $p(\lambda)$. Then we have the following commutative diagram
\[
  \begin{CD}
    1 @>>> S @>>> SL_n @>>> X @>>> 1\\
    @. @V V  V @V V V @| @.\\
    1 @>>> R_{K/\Q}(\G_m) @>>> GL_n @>>> X @>>> 1
  \end{CD}
\]
For any field extension $k/\Bbb Q$, one obtains the following commutative diagram

\[
  \begin{CD}
           @. 1 @. 1\\
           @. @VVV @VVV \\
    1 @>>> S(k) @>>> SL_n(k) @>>> X(k) @>\delta_1>> H^1(k,S)\\
    @. @V  V V @V V V @| @.\\
    1 @>>> (k\otimes_\Q K)^\times @>>> GL_n(k) @>>> X(k) @>>> 1\\
    @. @V  V N_{K/\Q} V @V V det V @.\\
    1 @>>> k^\times @= k^\times \\
    @. @V V\delta_2 V @V V V @.\\
    @. H^1(k,S) @. 1
  \end{CD}
\]
by Galois cohomology, Shapiro's Lemma and Hilbert 90.

 For any $x\in X(k)$, there is a $g\in GL_n(k)$ such that $gvg^{-1}=x$ by the second line of the above diagram. Let $y\in R_{K/\Q}(\G_m)(\bar k)$ such that $N_{K/\Q}(y)=\det(g)$. Then $gy^{-1}\in SL_n(\bar k)$ and $$(gy^{-1})v (gy^{-1})^{-1}=g(y^{-1} v y)g^{-1}=gvg^{-1}=x$$ and $$\delta_1(v)=\sigma(y^{-1}g)(y^{-1}g)^{-1}=\sigma(y^{-1})\sigma(g)g^{-1}y=\sigma(y^{-1})y=\delta_2(\det(g)^{-1}).$$
Therefore
   \begin{equation} \label{partial}  \aligned \delta_1: \  X(k)  & \longrightarrow H^1(k,S) \cong k^\times/N_{K/k}((K\otimes_\Bbb Q k)^\times) \\ x & \mapsto det(g)^{-1}\cdot N_{K/k}((K\otimes_\Bbb Q k)^\times) \endaligned \end{equation}

\medskip

Let $L$ be the Galois closure of $K/\Bbb Q$, $\Lambda=\Gal(L/\Bbb Q)$ and $\Upsilon=\Gal(L/K)$. Write $\widehat{S}=Hom(S,\Bbb G_m)$ be the character group of $S$. Then one has the short exact sequence of $\Lambda$-module
$$ 0 \longrightarrow \Bbb Z \longrightarrow Ind_\Upsilon^\Lambda(\Bbb Z) \longrightarrow \widehat{S}\longrightarrow 0 $$ which gives the long exact sequence
$$ \rightarrow H^1(\Lambda,Ind_\Upsilon^\Lambda(\Bbb Z)) \rightarrow H^1(\Lambda, \widehat{S})\rightarrow H^2(\Lambda,\Bbb Z)\rightarrow H^2(\Lambda, Ind_\Upsilon^\Lambda(\Bbb Z)) $$
where $$H^1(\Lambda,Ind_\Upsilon^\Lambda(\Bbb Z))=H^1(\Upsilon,\Bbb Z)=0$$ and $$H^2(\Lambda,\Bbb Z)=H^1(\Lambda,\Bbb Q/\Bbb Z)=Hom(\Lambda,\Bbb Q/\Bbb Z)$$ and $$H^2(\Lambda, Ind_\Upsilon^\Lambda(\Bbb Z))=H^2(\Upsilon, \Bbb Z)=H^1(\Upsilon, \Bbb Q/\Bbb Z)=Hom(\Upsilon, \Bbb Q/\Bbb Z)$$ by Shapiro's Lemma. Since
$$ 0\rightarrow H^1(\Lambda,\widehat{S})\xrightarrow{inf} H^1(\Bbb Q, \widehat{S})\xrightarrow{res} H^1(L,\widehat{S})$$ and
$$ H^1(L,\widehat{S})=Hom(\Gal(\bar{\Bbb Q}/L),\widehat{S})=0 $$ where $\bar{\Bbb Q}$ is an algebraic closure of $\Bbb Q$, one concludes that
$$Pic(S)= H^1(\Bbb Q, \widehat{S})=H^1(\Lambda,\widehat{S})=ker(\widehat{\Lambda}\rightarrow \widehat{\Upsilon})  $$ by Theorem of \S 4.3 of Chapter 2 in \cite{Vo}, where $$\widehat{\Lambda}= Hom(\Lambda,\Bbb Q/\Bbb Z) \ \ \ \text{ and } \ \ \ \widehat{\Upsilon}=Hom(\Upsilon, \Bbb Q/\Bbb Z) . $$

By Proposition 2.10 of \cite{CTX}, one only needs to consider that $Pic(S)$ is not trivial. For any $\chi \in ker(\widehat{\Lambda}\rightarrow \widehat{\Upsilon})$ with $\chi\neq 1$, one gets a non-trivial abelian extension $A/\Bbb Q$ inside $L/\Bbb Q$ such that $ker(\chi)=\Gal(L/A)$. By Minkowski's Theorem (see (2.17) of Chapter III in \cite{N}), there is a prime $p$ such that $p$ is ramified in $A/\Bbb Q$. By (1.7) Proposition of Chapter V and (5.6) Proposition of Chapter VI in \cite{N}, there is an idele $(u_l)_l$ of $\Bbb Q$ defined by $u_l=1$ for $l\neq p$ and $u_p\in \Bbb Z_p^\times$ for $l=p$ such that
$$ \chi (\psi_{A/\Bbb Q}((u_l)_l)) \neq 1$$ where $\psi_{A/\Bbb Q}$ is the Artin map.

Let $\xi\in Br(X)$ be the image of $\chi$ under the map $$\delta_{tors}(SL_n): Pic(S)\rightarrow Br(X)$$ defined in P.314 of \cite{CTX}. By Proposition 2.10 in \cite{CTX}, one has that $\xi$ is not trivial in $Br(X)/Br(\Bbb Q)$. Applying the change of variables on $X(\Bbb Q_p)$ by using the diagonal matrix $diag(u_p^{-1},1,\cdots,1)$ $$(x_{i,j})\mapsto diag(u_p^{-1},1,\cdots,1)\cdot (x_{i,j}) \cdot diag(u_p, 1,\cdots,1)$$ which leaves ${\bf {X}}(\Bbb Z_p)$ stable, one obtains that
$$\int_{{\bf {X}}(\Bbb Z_p)} \xi d_p=\chi (\psi_{A/\Bbb Q}((u_l)_l ))\int_{{\bf {X}}(\Bbb Z_p)} \xi d_p
$$ by (\ref{partial}) and the diagram (3.1) in \cite{CTX}. Therefore
$$\int_{{\bf {X}}(\Bbb Z_p)} \xi d_p=0$$ and the result follows from Theorem \ref{semisimple} and Proposition 2.10 (ii) in \cite{CTX}.
\end{proof}

We will answer a question raised by Borovoi how to compute the ratio of the number of the integral solutions with the Hardy-Littlewood expectation for Example 6.3 in \cite{BR95}.

\begin{exa} Let $a\in \Bbb Z$ with $a\neq 0$ and ${\bf X}_a$ be a scheme over $\Bbb Z$ defined by the equation
$$ det(x_{i,j})_{n\times n} =a \ \ \ \text{ with  $x_{i,j}=x_{j,i}$ \ for all $1\leq i,j\leq n$ } $$ with $X_a={\bf X}_a\times_{\Bbb Z} \Bbb Q$ and $n\geq 3$. Define $$c_n(a)= \lim_{T\rightarrow \infty} \frac{N({\bf X}_a,T)}{(\prod_{p<\infty} \int_{{\bf X}_a(\Bbb Z_p)} d_p)\cdot \int_{{\bf X}_a(\Bbb R, T)} d_\infty } .$$ Then
$$ c_n(a)=1+ \prod_{p|2a}\frac{\int_{{\bf X}_a(\Bbb Z_p)} h \ d_p}{ \int_{{\bf X}_a(\Bbb Z_p)} d_p} \cdot \lim_{T\rightarrow \infty} \frac{\int_{{\bf X}_a(\Bbb R, T)} h  \ d_\infty }{\int_{{\bf X}_a(\Bbb R, T)} d_\infty}$$ where $h: X(\Bbb Q_p)\longrightarrow \{\pm1\}$ is the Hasse-Witt symbol function defined in P.167 of \cite{OM}. In particular   $$c_n(a)= \begin{cases} 1 \ \ \ & \text{if $a>0$ and $n\equiv 2 \mod 4$} \\
\frac{1}{2} \ \ \ & \text{if $a=1$ and $n=3$.} \end{cases} $$
\end{exa}

\begin{proof} Since $SL_n$ acts on $X_a$ by $ x\circ g = g'xg$ where $g'$ is the transpose of $g$, one has that $X_a$ is the homogeneous space of $SL_n$ with a rational point of diagonal matrix $v= diag(a,1,\cdots,1)\in {X_a}(\Bbb Q)$ and the stabilizer of $v$ is the special orthogonal group $SO(v)$ defined by $diag(a,1,\cdots,1)$. Moreover,  one gets that $$Pic(SO(v))\cong \Bbb Z/(2)$$ by Proposition 2.5 and 2.6 in \cite{CTX}. The non-trivial element $E$ of $Pic(SO(v))$ gives the central extension of $SO(v)$ by $\Bbb G_m$ satisfying the following diagram
\[ \begin{CD}
1 @>>>  \mu_2 @>>> Spin @>>> SO(v) @>>> 1  \\
       @.  @VVV  @VVV @V{id}VV @.  \\
1 @>>>  \Bbb G_m @>>> E @>>> SO(v) @>>>  1
\end{CD} \]
because $Spin$ is almost simple. Applying the Galois cohomology, one has
\begin{equation} \label{hw} \begin{CD}
H^1(k, SO(v))  @>{h'}>> _2Br(k)  \\
        @V{id}VV @VVV  \\
H^1(k, SO(v)) @>>>  Br(k)
\end{CD} \end{equation}
for any field extension $k/\Bbb Q$, where $h'$ is the Hasse-Witt invariant by (31.41) in \cite{KMRT}. By Proposition 2.10 in \cite{CTX}, one has that $\delta_{tors}(SL_n)(E)$ is the non-trivial element $\xi$ of $Br(X_a)/Br(\Bbb Q)$. Moreover, the evaluation of $\xi$ over $X_a(k)$ is equal to the Hasse-Witt invariant of the corresponding torsor by Proposition 2.9 in \cite{CTX} and the above commutative diagram (\ref{hw}). By the commutative diagram (3.1) in \cite{CTX} and Theorem \ref{semisimple} and the Hilbert reciprocity law (see Chapter VII in \cite{OM}), one concludes that $$ c_n(a)=1+  \frac{(\prod_{p<\infty} \int_{{\bf X}_a(\Bbb Z_p)} h \ d_p)}{(\prod_{p<\infty} \int_{{\bf X}_a(\Bbb Z_p)} d_p)} \cdot \lim_{T\rightarrow \infty} \frac{\int_{{\bf X}_a(\Bbb R, T)} h  \ d_\infty }{\int_{{\bf X}_a(\Bbb R, T)} d_\infty}$$ where $h: X_a(\Bbb Q_p)\longrightarrow \{\pm1\}$ is the Hasse-Witt symbol function. If $p\nmid 2a$, then $h(x)=1$ by 92:1 in \cite{OM} and the first part of the result follows.

 There are only finitely many orbits of $SL_n(\Bbb R)$ inside $X_a(\Bbb R)$ classified by the signatures. Among these orbits, the most significant orbits $Y_{+}(\Bbb R)$ and $Y_{-}(\Bbb R)$ are given by the exact sequence
$$ 1\longrightarrow SO_{\pm} \longrightarrow SL_n(\Bbb R) \xrightarrow{\pi_\pm} Y_\pm(\Bbb R) \longrightarrow 1 $$
where $SO_\pm$ are the special orthogonal groups defined by $\pm I_n$ respectively. Both $SO_{\pm}$ are the compact subgroups of $SL_n(\Bbb R)$. The rest orbits $Y_i$ are given by
$$ 1\longrightarrow SO_i \longrightarrow SL_n(\Bbb R) \xrightarrow{\pi_i} Y_i (\Bbb R) \longrightarrow 1 $$
where $SO_i$ are the special orthogonal groups defined by the rest of representatives of the orbits of $SL_n(\Bbb R)$ whose signatures are different from those of $\pm I_n$. Then such $SO_i$'s are not compact. Since
$$\int_{{\bf X}_a(\Bbb R, T)\cap Y_\pm(\Bbb R)} d_\infty = \frac{Vol(\pi_\pm^{-1}({\bf X}_a(\Bbb R, T)\cap Y_\pm(\Bbb R)))}{Vol(\pi_\pm^{-1}({\bf X}_a(\Bbb R, T)\cap Y_\pm(\Bbb R))\cap SO_\pm) }$$
and $$ \int_{{\bf X}_a(\Bbb R, T)\cap Y_i(\Bbb R)}d_\infty
= \frac{Vol(\pi_i^{-1}({\bf X}_a(\Bbb R, T)\cap Y_i(\Bbb R)))}{Vol(\pi_i^{-1}({\bf X}_a(\Bbb R, T)\cap Y_i(\Bbb R))\cap SO_i )},  $$
where
$$ \lim_{T\rightarrow \infty} Vol(\pi_\pm^{-1}({\bf X}_a(\Bbb R, T)\cap Y_\pm(\Bbb R))\cap SO_\pm) = Vol(SO_\pm) <\infty$$ and
$$ \lim_{T\rightarrow \infty} Vol(\pi_i^{-1}({\bf X}_a(\Bbb R, T)\cap Y_i(\Bbb R))\cap SO_i) = \infty ,$$ one concludes that
$$\lim_{T\rightarrow \infty} \frac{\int_{{\bf X}_a(\Bbb R, T)} h  \ d_\infty }{\int_{{\bf X}_a(\Bbb R, T)} d_\infty}=
\begin{cases} 1  \ \ \ & \text{if $a>0$ and $n\equiv 1 \mod 2$} \\
 1 \ \ \ & \text{if $a<0$ and $n\equiv 3 \mod 4$} \\
 -1 \ \ \ & \text{if $a<0$ and $n\equiv 1 \mod 4$} \\
 1 \ \ \ & \text{if $a>0$ and $n\equiv 0 \mod 4$} \\
 0 \ \ \ & \text{if $a>0$ and $n\equiv 2 \mod 4$}
 \end{cases}$$
by computation of Hasse-Witt invariant over $\Bbb R$.

For $a=1$ and $n=3$, one has
$$c_3(1)= 1+ \frac{\int_{{\bf X}_1(\Bbb Z_2)} h \ d_2}{ \int_{{\bf X}_1(\Bbb Z_2)} d_2} .$$ By 93:18 (iv) in \cite{OM}, there are two orbits of ${\bf X}_1(\Bbb Z_2)$ under the action of $SL_3(\Bbb Z_2)$ with the representatives
$$ L_{-1}=\left(\begin{array} [c]{lll}
1 \ \ & 0 \ \ & 0   \\ 0 \ \ & -1  \ \ & 0  \\
0 \ \ & 0  \ \ & -1
\end{array}
\right)  \ \ \ \text{and} \ \ \ L_1= \left(\begin{array} [c]{lll}
2 \ \ & 1 \ \ & 0   \\ 1 \ \ & 2  \ \ & 0  \\
0 \ \ & 0  \ \ & 3^{-1}
\end{array}
\right) $$
where $h(L_{-1})=-1$ and $h(L_1)=1$. Therefore
$$ \int_{{\bf X}_1(\Bbb Z_2)} h \ d_2= \frac{Vol(SL_3(\Bbb Z_2))}{Vol(SO(L_1))}-\frac{Vol(SL_3(\Bbb Z_2))}{Vol(SO(L_{-1}))}$$ and
$$ \int_{{\bf X}_1(\Bbb Z_2)}  \ d_2= \frac{Vol(SL_3(\Bbb Z_2))}{Vol(SO(L_1))}+\frac{Vol(SL_3(\Bbb Z_2))}{Vol(SO(L_{-1}))}. $$ By Lemma 1.8.1 in \cite{BR95}, one has  $$\beta_2(L_1, L_1)=Vol(SO(L_1)) \ \ \ \text{and} \ \ \ \beta_2(L_{-1}, L_{-1})=Vol(SO(L_{-1})) $$ where $\beta_2(L_1, L_1)$ and $\beta_2(L_{-1}, L_{-1})$ are local densities in sense of \S 5.6 of Chapter 5 \cite{Ki}. Since
$$ \frac{\beta_2(L_1,L_1)}{\beta_2(L_{-1},L_{-1})} = \frac{1+2^{-1}}{1-2^{-1}}=3$$
by Theorem 5.6.3 in \cite{Ki}, one has
$$ c_3(1)= 1+ \frac{1-\frac{\beta_2(L_1,L_1)}{\beta_2(L_{-1},L_{-1})}}{1+\frac{\beta_2(L_1,L_1)}{\beta_2(L_{-1},L_{-1})} }=\frac{1}{2} . $$
\end{proof}
\bigskip

\bf{Acknowledgment} \it{Part of the work was done when the second author visited University of Nancy in April, 2012. He would like to thank Wu Jie for his kind invitation and a lot of helpful discussion. The first author is supported by NSFC, grant \#
10671104, 973 Program 2013CB834202 and grant DE 1646/2-1 of the Deutsche Forschungsgemeinschaft. The second author is supported by NSFC, grant \#
10325105 and \# 10531060. }


\end{document}